\documentclass{paper}

\usepackage{arxiv}

\usepackage[utf8]{inputenc} 
\usepackage[T1]{fontenc}    
\usepackage{hyperref}       
\usepackage{url}            
\usepackage{booktabs}       
\usepackage{amsfonts}       
\usepackage{nicefrac}       
\usepackage{microtype}      
\usepackage{subcaption}
\usepackage{mathrsfs,bm,siunitx}
\usepackage{comment}
\usepackage{appendix}
\usepackage{lipsum}
\usepackage{graphicx}
\usepackage{amsmath,amsthm}
\numberwithin{equation}{section}
\usepackage{color}
\usepackage{colortbl}
\usepackage{xcolor}
\usepackage{transparent}
\usepackage{todonotes}

\graphicspath{{}}
\usepackage{tikz}
\usetikzlibrary{arrows}
\newtheorem{theorem}{Theorem}[section]
\newtheorem{proposition}{Proposition}[section]

\newtheorem{corollary}{Corollary}[section]
\theoremstyle{definition}
\newtheorem{dfn}{Definition}[section]
\newtheorem{remark}{Remark}[section]

\title{Well-defined forward operators in dynamic diffractive tensor tomography using viscosity solutions of transport equations}

\author{
  Lukas Vierus \\
  Department of Numerical Mathematics\\
  Saarland University\\
  Saarbr\"ucken, Germany \\
  \texttt{vierus@num.uni-sb.de} \\
       \And
  Thomas Schuster \\
  Department of Numerical Mathematics\\
  Saarland University\\
  Saarbr\"ucken, Germany \\
  \texttt{thomas.schuster@num.uni-sb.de} \\
}

\begin{document}
\maketitle
\begin{abstract}
We consider a general setting for dynamic tensor field tomography in an inhomogeneous refracting and absorbing medium as inverse source problem for the associated transport equation. Following Fermat's principle the Riemannian metric in the considered domain is generated by the refractive index of the medium. There is wealth of results for the inverse problem of
recovering a tensor field from its longitudinal ray transform in a static euclidean setting, whereas there are only
few inversion formulas and algorithms existing for general Riemannian metrics and time-dependent tensor fields.
It is a well-known fact that tensor field tomography is equivalent to an inverse source problem for a transport equation
where the ray transform serves as given boundary data. We prove that this result extends to the dynamic case.
Interpreting dynamic tensor tomography as inverse source problem represents a holistic approach in this field.
To guarantee that the forward mappings are well-defined, it is necessary to prove existence and uniqueness for
the underlying transport equations. Unfortunately, the bilinear forms of the associated weak formulations do not 
satisfy the coercivity condition. To this end we transfer to viscosity solutions and prove their unique existence
in appropriate Sobolev (static case) and Sobolev-Bochner (dynamic case) spaces under a certain assumption
that allows only small variations of the refractive index. Numerical evidence is given that
the viscosity solution solves the original transport equation if the viscosity term turns to zero.
\end{abstract}


\smallskip


\smallskip

\section{Introduction}


Tensor field tomography (TFT) means to determine a tensor field, or at least part of it, from given integral data along geodesic curves of a Riemannian metric: the so-called ray transform of the field. In this article we consider TFT in a very general setting for static as well as for time-dependent fields and in a medium with absorption and which is inhomogeneous. The latter property is mathematically modeled by the fact that the domain under consideration is equipped with a corresponding Riemannian metric whose geodesics correspond to the integration paths of the ray transform. Especially
if we use, e.g., ultrasound measurements for data acquisition and follow Fermat's principle, then the metric is generated by the refractive index and the geodesic curves are normal to the propagating wave fronts. In this article we restrict the
Riemannian metric to this setting.

TFT has many possible applications. One is the reconstruction of velocity fields of liquids and gases. This can be used, e.g., to represent blood flows in medicine. TFT is also used in electron tomography, industry, geo- and astrophysics to name only a few application fields. Pioneered by Norton \cite{norton1988} in 1988, fundamental results on Doppler tomography followed by Juhlin \cite{juhlin1992}, Gullberg \cite{gullberg}, Schuster \cite{schuster2008}, and Strahlen\cite{strahlen1998}. A singular value decomposition for the 2D ray transform for vector fields can be found in \cite{DEREVTSOV;ET;AL:11}. Prince \cite{PRINCE:96} used vector tomography in MRI and Panin et al \cite{panin2002} in diffusion tensor MRI. In Sharafutdinov \cite{sharafutdinov2007}, procedures for tomography with limited data can be found. 

For a tensor field $f$ of rank $m>1$ in a Riemannian domain $(M,g)$ the attenuated longitudinal ray transform is defined as

\begin{align*}
[\mathcal{I}_\alpha f](p,q) = \int_{\gamma_{pq}} f(x)\cdot \dot{\gamma}_{pq}(x)
\exp\left(\int_{\gamma_{xq}}\alpha(t)\mathrm{d}\sigma (t)\right)\mathrm{d}\sigma (x),
\end{align*}

where $\gamma_{pq}$ is a geodesic curve connecting two points $p,q\in \partial M$ and $\alpha\geq 0$ denotes the absorption coefficient. The inverse problem of TFT is to determine $f$ from the knowledge of $\mathcal{I}_\alpha f$ on a subset $S\subset (\partial M\times \partial M)$. It can be shown, see, e.g., \cite{paternain2012, Sharafutdinov_1994}, that this problem is equivalent to computing the source term $f$ in the transport equation

\begin{align*}
\mathcal{H} u (x,\xi) + \alpha u (x,\xi) = f\cdot \xi^m,
\end{align*}

where $\mathcal{H}$ denotes the geodesic vector field corresponding to the metric $g$, $\xi\in T_x M$ is a tangent vector in $x$ and $\xi^m= \xi\otimes \cdots \otimes\xi$ is the $m$-fold tensor product of $\xi$. The ray transform $\mathcal{I}_\alpha f$ determines
the given boundary data of $u$. This formulation offers the possibility for a holistic approach to TFT in general settings,
i.e., taking absorption and refraction into account. It even extends to dynamic settings of the ray transform for 
time-depending tensor fields $f$,
\begin{align*}
[\mathcal{I}_{\alpha}^d f](t, x,\xi)=\int_{\tau_{-}(x,\xi)}^{0}\langle f(t+\tau, \gamma_{x,\xi}(\tau)),\dot{\gamma}_{x,\xi}^m (\tau)\rangle \exp\left(-\int_{\tau}^{0}\alpha(\gamma_{x,\xi}(\sigma),\dot{\gamma}_{x,\xi}(\sigma))\mathrm{d}\sigma\right) \mathrm{d}\tau.\\
\end{align*}
as we will show.


So far, only little research has been done on tensor tomography taking refraction into account. 
Among the applications of tensor tomography are diffraction tomography of deformations \cite{lionheart2015}, polarization tomography of quantum radiation \cite{karassiov2004} and the tomography of tensor fields of stresses of e.g. fiberglass composites \cite{puro2007}. In addition, there are polarization tomography \cite{Sharafutdinov_1994}, plasma diagnosis \cite{balandin} and photoelasticity \cite{puro2016}. Furthermore, novel methods exist, which are especially successful in biology and medicine. These include diffusion MRI tomography, which can be used to study the brain in detail. On the other hand, cross-polarized optical coherent tomography allows for a detailed examination of cells and is used for the diagnosis of cancer \cite{panin2002}. Due to the fact that the reconstruction of a tensor field of rank $m>1$ using one-dimensional data $\mathcal{I}_\alpha f$, $\mathcal{I}_\alpha^d f$ is obviously underdetermined, the ray transform must have a non-trivial null space. Decompositions of symmetric 2D tensor fields exist \cite{derevtsov_2015}, so it is possible to reconstruct them uniquely from longitudinal and transverse ray transforms. For higher dimensions there are no such decompositions yet, however one can define the mixed ray transforms in arbitrary dimensions \cite{Sharafutdinov_1994}.

In the publications mentioned above, Euclidean geometry is assumed. 
In \cite{uhlmann} and \cite{uhlmann2} tomography for refractive media is studied for the special case of scalar fields in a 2D domain. There, questions about the range of the ray transform as well as uniqueness and stability of the solution are studied. Results on vector and tensor tomography in Riemannian manifolds can be found in 
\cite{udoschuster, sharafutdinov1992curvature, STEFANOV;UHLMANN:04}. In \cite{STEFANOV;UHLMANN;VASY:14} the author prove local invertibility of the geodesic ray transform for tensor fields of order 1 and 2 near a strictly convex boundary point and present a reconstruction formula. A study on the influence of refraction to the reconstruction accuracy can be found in \cite{DEREVTSOV;ET;AL:00}.

Dynamic tomographic problems arise, e.g., in medical imaging, where artifacts caused by motion are to be corrected. The dynamic inverse problems can be regularized by a Tikhonov-Phillips method (c.f. \cite{louis2002theo}, \cite{louis2002appl}) or the method of approximate inverse \cite{hahn2015}. Motion compensation strategies are also investigated in \cite{BLANKE;ET;AL:20, hahn2014, hahn2017}. In \cite{hahnlouis2017} the relation between motion and resolution has been investigated. 

\textit{Our contributions:} We first prove that the integral representations $\mathcal{I}_{\alpha} f$, $\mathcal{I}_{\alpha}^d f$ satisfy specific boundary-, respectively initial-boundary-value problem for transport equations. Subsequently we investigate existence and uniqueness of weak solutions. It will be shown that these problems lack of uniqueness since the corresponding bilinear form is not $H^1$-coercive. As a remedy we turn over to viscosity solutions for which we are able to prove unique existence under a certain, mild additional condition to the refractive index $n(x)$. Numerical evaluations show that the viscosity solutions are appropriate approximations to original solutions. The results are of utmost importance for solving tensor tomography problems in fairly general settings, since such problems can be interpreted as inverse source problems for transport equations. The results of this article then imply that the corresponding forward mappings are well-defined.


\section{Geodesic differential equation and ray transforms on a CDRM}


First we want to model our problem. For this purpose we define corresponding spaces (c.f. \cite{uhlmann},\cite{uhlmann2}) and specify how the course of a ray within a medium can be inferred unambiguously on the basis of the refractive index.

According to Fermat's principle a signal propagates along the path with shortest travel time. This implies that we are able to interpret the ray as geodesic curve associated with a Riemannian metric which is generated by the refractive index $n(x)$. Allover this manuscript we assume that 

\[  n(x) \geq c_n > 0\qquad \mbox{a.e.}   \]

for a positive constant $c_n$. Let $\gamma:[a,b]\rightarrow \mathbb{R}^3$ be a smooth curve. The time a signal needs to propagate from its initial point $\gamma(a)$ to to $\gamma(b)$ is given by

\begin{align*}
T(\gamma(a),\gamma(b))&= \int_\gamma n(x)\mathrm{d}\sigma(x)=\int_a^b n(\gamma(t))||\dot{\gamma}(t)||\mathrm{d}t=\int_a^b \sqrt{n^2(\gamma(t))||\dot{\gamma}(t)||^2}\mathrm{d}t=\int_\gamma \mathrm{d}s,\\
\end{align*}

where $n(x)$ is the refractive index of the considered medium and $\mathrm{d}s^2 = n^2(x)||\mathrm{d}x||^2$ is the length element of the Riemannian metric $g$ with
\begin{align}\label{g_n}
g_{ij}(x)=n^2(x)\delta_{ij},\quad 1\leq i,j\leq 3.
\end{align}
Here we used Einstein's summation convention meaning that we sum up over double indices. In $(\mathbb{R}^3,g)$ for a given function $f:\mathbb{R}^3\rightarrow\mathbb{R}$ the gradient reads as $\nabla f = g^{ij}\partial_i f \partial_j = n^{-2}(x) \nabla_{eucl}f$, where $g^{ij}$ are the entries of the inverse of $g_{ij}$, and for tangential vectors $u,v\in \mathbb{R}^3$ the inner product is given as $\langle u,v\rangle  = g_{ij}u_i v_j$ and thus $||u|| = n(x)||u||_{eucl}$. For details we refer to \cite{petersen}.

A curve $\gamma$ minimizing $T(\gamma(a),\gamma(b))$ is a geodesic of $g$. Such a curve satisfies the $\textit{geodesic differential equation}$ given by

\begin{align*}
\ddot{\gamma}_k+ \Gamma_{ij}^k(\gamma) \dot{\gamma}_i \dot{\gamma}_j = 0.\\
\end{align*}

The Christoffel symbols $\Gamma_{ij}^{k}$ are defined by 
\begin{align*}
\Gamma_{ij}^{k}(x)&=\frac{1}{2}g^{kp}(x)\left(\frac{\partial g_{ip}}{\partial x^j}(x)+\frac{\partial g_{jp}}{\partial x^i}(x)-\frac{\partial g_{ij}}{\partial x^p}(x)\right).\\
\end{align*}

In case of the metric \eqref{g_n} we compute

\begin{align}
\Gamma_{ij}^{k}(x)= n^{-1}(x)\left( \frac{\partial n}{\partial x_j} (x) \delta_{ik} + \frac{\partial n}{\partial x_i} (x) \delta_{jk} - \frac{\partial n}{\partial x_k}(x)\delta_{ij}\right).\label{use}\\\nonumber
\end{align}

Initializing at time $t=0$ the starting point and tangential vector, i.e. setting $\gamma(0)=x$, $\dot{\gamma}(0)=\xi$ and denoting such a $\gamma$ as $\gamma=\gamma_{x,\xi}$ gives the following theorem:

\begin{theorem}
Let $(M,g)$ be a compact Riemannian manifold in $\mathbb{R}^3$ and $n\in C^2(M)$. Then the following initial value system has a unique solution:

\begin{align*}
\ddot{\gamma}_k+ \Gamma_{ij}^k(\gamma) \dot{\gamma}_i \dot{\gamma}_j = 0,\qquad \gamma(0)=x,\dot{\gamma}(0)=\xi.
\end{align*}
\end{theorem} 

\begin{proof}
The proof works similarly to the one in \cite{udoschuster} for two dimensions. First, we write the second-order ODE into a system of first-order ODEs. 
We set $\Gamma_{ij}(x)=(\Gamma_{ij}^1(x), \Gamma_{ij}^2(x), \Gamma_{ij}^3(x))$ and obtain

\begin{align*}
\left\{\begin{array}{ll}
	\frac{\mathrm{d}}{\mathrm{d}t}\gamma(t)&= \dot{\gamma}(t) \\
	\frac{\mathrm{d}}{\mathrm{d}t}\dot{\gamma}(t)&= - \Gamma_{ij} (\gamma(t))\dot{\gamma}_i(t)			    \dot{\gamma}_j(t)\\
    \gamma(0)&=x\\
    \dot{\gamma}(0)&=\xi   
    	\end{array}\right. .
\end{align*}

 
According to Picard-Lindel\"of's theorem (c.f. \cite{walter}) and the mean value theorem, it is sufficient to show that the gradient of $(\dot{\gamma}(t), - \Gamma_{ij} (\gamma(t))\dot{\gamma}_i(t)			    \dot{\gamma}_j(t))$ with respect to $z=(\gamma,\dot{\gamma})$ remains bounded. Obviously we have

\begin{align*}
\frac{\mathrm{d}}{\mathrm{d}\gamma_l}\dot{\gamma}_k=0,\qquad \frac{\mathrm{d}}{\mathrm{d}\dot{\gamma}_l}\dot{\gamma}_k = \delta_{kl},\qquad k,l=1,2,3.
\end{align*}

Next, we divide the sum in \eqref{use} into the following cases:



\begin{align*}
i=j=k && i=k\neq j\\
i=j\neq k && i\neq j=k.\\
\end{align*}

The case where all indices are different vanishes and can be neglected. We obtain for $k=1,2,3$

\begin{align}\label{christ_refractive}
-\Gamma_{ij}^{k}(\gamma)\dot{\gamma}_i \dot{\gamma}_j 
&= -n^{-1}(\gamma)\left(\frac{\partial n}{\partial x_k} (\gamma)\dot{\gamma}_k^2 + \sum_{i=k\neq j}\frac{\partial n}{\partial x_j}(\gamma)\dot{\gamma}_i \dot{\gamma}_j + \sum_{i\neq k}(-\frac{\partial n}{\partial x_k}(\gamma))\dot{\gamma}_i^2 + \sum_{i\neq j = k} \frac{\partial n}{\partial x_i}(\gamma) \dot{\gamma}_i \dot{\gamma}_j\right)\nonumber\\
&=  -n^{-1}(\gamma)\left(2\frac{\partial n}{\partial x_k}(\gamma)\dot{\gamma}_k^2 + 2\sum_{i=k\neq j}\frac{\partial n}{\partial x_j}(\gamma)\dot{\gamma}_i \dot{\gamma}_j + \sum_{i}\left(-\frac{\partial n}{\partial x_k}(\gamma)\right)\dot{\gamma}_i^2\right)\nonumber\\
&= n^{-1}(\gamma)\left(n^{-2}(x)\frac{\partial n}{\partial x_k}(\gamma)||\dot{\gamma}|| ^2 - 2\frac{\partial n}{\partial x_k}(\gamma)\dot{\gamma}_k^2 - 2\sum_{j\neq k} \frac{\partial n}{\partial x_j}(\gamma)\dot{\gamma}_k \dot{\gamma}_j\right)\nonumber\\
&= n^{-1}(\gamma)\left(n^{-2}(x)\frac{\partial n}{\partial x_k}(\gamma)||\dot{\gamma}|| ^2  - 2\sum_{j} \frac{\partial n}{\partial x_j}(\gamma)\dot{\gamma}_k \dot{\gamma}_j\right)\nonumber\\
&= n^{-1}(\gamma)\left(n^{-2}(\gamma)\frac{\partial n}{\partial x_k} (\gamma)||\dot{\gamma}|| ^2  - 2\dot{\gamma}_k \langle\nabla n(\gamma), \dot{\gamma}\rangle \right).\\\nonumber
\end{align}

Thus we get for $l=1,2,3$

\begin{align*}
\frac{\mathrm{d}}{\mathrm{d}\dot{\gamma}_l}(-\Gamma_{ij}^{k}(\gamma)\dot{\gamma}_i \dot{\gamma}_j)
&= n^{-1}(\gamma)\left( 2n^{-2}(\gamma)\frac{\partial n}{\partial x_k}(\gamma)\dot{\gamma}_l - 2\dot{\gamma}_k  \frac{\partial n}{\partial x_l} (\gamma) -2\delta_{k,l} \langle \nabla n(\gamma),\dot{\gamma}\rangle \right)\\
\frac{\mathrm{d}}{\mathrm{d}\gamma_l}(-\Gamma_{ij}^{k}(\gamma)\dot{\gamma}_i \dot{\gamma}_j)
&= -\frac{\partial n}{\partial x_l}(\gamma) n^{-2}(\gamma)\left(3n^{-2}(x)\frac{\partial n}{\partial x_k}(\gamma)||\dot{\gamma}|| ^2  - 2\dot{\gamma}_k \langle\nabla n(\gamma), \dot{\gamma}\rangle  \right)
\\&+ n^{-1}(\gamma)\left( n^{-2}(x)\frac{\partial^2 n}{\partial x_k \partial x_l}(\gamma) ||\dot{\gamma}|| ^2 - 2\dot{\gamma}_k \langle \nabla \left(\frac{\partial n}{\partial x_l}(\gamma) \right), \dot{\gamma}\rangle    \right).\\
\end{align*}

Since $n>0$ and all the derivatives are bounded it follows the asserted statement.
\end{proof}


Hence, waves in $(M,g)$ with a smooth refractive index propagate along geodesics that are uniquely defined by the initial point and direction. For completeness we state the definition of a compact dissipative Riemannian manifold.

\begin{dfn}
Let $M\subset \mathbb{R}^d$ be a compact manifold and $g$ a Riemannian metric with strictly convex boundary $\partial M$.
If for every given point $x\in M$ and non-zero vector $\xi$ in its tangent space $T_x M$
the geodesic $\gamma_{x,\xi}(t)$ cannot be extended further than to a finite interval $[\tau_{-}(x,\xi),\tau_{+}(x,\xi)]$, 
then we call $(M,g)$ a \textit{compact dissipative Riemannian manifold} (CDRM).
\end{dfn}

In a CDRM all geodesics have a finite length. The interval limits can be characterized by

\begin{align*}
\tau_{-}(x,\xi)&=\max \lbrace \tau\in (-\infty,0]: \gamma_{x,\xi}(t)\cap \partial M \neq \emptyset\rbrace\\
\tau_{+}(x,\xi)&=\min \lbrace \tau\in [0, \infty): \gamma_{x,\xi}(t)\cap \partial M \neq \emptyset\rbrace.\\
\end{align*}

Hence, $\gamma_{x,\xi}(\tau_{\mp}(x,\xi))$ are entry and exit point of a geodesic that is for $\tau=0$ at position $x$ and moves in direction $\xi\in T_x M$.

\begin{figure}[ht]
	\centering
  \includegraphics[width=0.6\textwidth]{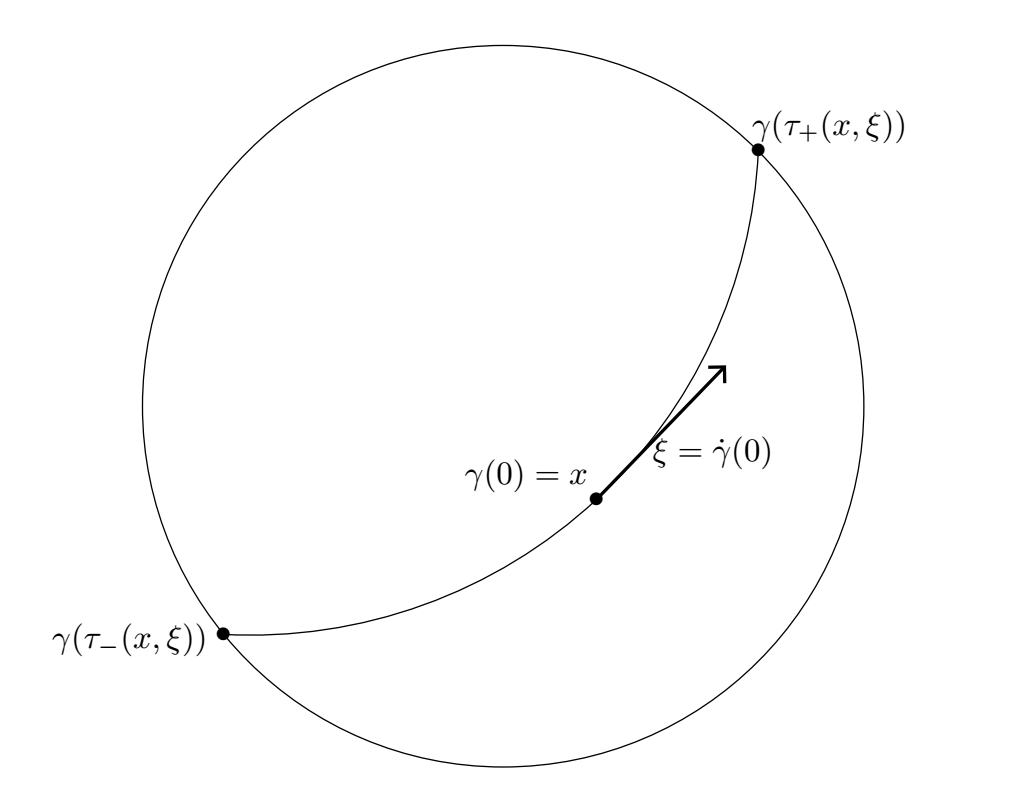}
	\caption{Sketch of a geodesic curve with parametrization}
	\label{fig1}
\end{figure}

We denote the tangent bundle of the manifold $M$ by

\begin{align*}
TM=\lbrace(x,\xi)|x\in M, \xi\in T_x M\rbrace
\end{align*}

and its submanifold consisting of unit vectors by

\begin{align*}
\Omega M=\lbrace (x,\xi)\in TM| ||\xi||  = 1\rbrace.
\end{align*}

Furthermore we introduce

\begin{align*}
T^0 M&=\lbrace (x,\xi)\in TM| \xi\neq 0\rbrace,\\
\partial_{\pm}\Omega M &= \lbrace (x,\xi)\in \Omega M|x\in \partial M; \pm\langle \xi,\nu(x)\rangle\geq 0\rbrace.
\end{align*}

Note that $\partial_{+} \Omega M$ and $\partial_{-} \Omega M$ are compact manifolds and

\begin{align*}
\partial_{+} \Omega M\cap\partial_{-} \Omega M = \Omega M\cap T(\partial M).
\end{align*}

Using the implicit function theorem and the strict convexity of $M$ implies that $\tau_{\pm}$ are smooth on $T^{0}M\backslash T(\partial M)$.




Without loss of generality we assume that $f$ is supported in the unit sphere and set $M:=\lbrace x\in \mathbb{R}^3: ||x||_{eucl}\leq 1\rbrace$.
Given an integer $m\geq 0$, we denote by $S^m$ the space of all functions $\underbrace{\mathbb{R}^3\times\dots\times\mathbb{R}^3}_{m\textit { factors}}\rightarrow \mathbb{R}$ that are $\mathbb{R}$-linear and invariant with respect to all transpositions of the indices. Moreover, we define $\tau_M=(TM, p, M)$ as the tangent bundle and $\tau_{M}'=(T'M, p', M)$ as the cotangent bundle on $M$, where $p:TM\to M$, $p':T'M\to M$ are corresponding projections to $M$, $M'$, respectively. For nonnegative integers $r$ and $s$, we set $\tau_s^r M= (T_s^r M, p_s^r, M)$ as the vector bundle defined by 

\begin{align*}
\tau_r^s M=\underbrace{\tau_{M}\otimes \dots \tau_M}_{r \textrm{ times}} \otimes \underbrace{\tau_{M}'\otimes\dots \otimes \tau_{M}'}_{s \textrm{ times}}.\\
\end{align*}

We denote the subbundle of $\tau_{m}^{0} M$ consisting of all tensors that are symmetric in all arguments by $S^m \tau_M '$. 

\begin{dfn}\label{D-ART}
For given $\alpha\in L^{\infty}(\Omega M)$ we define the $\textit{attenuated ray transform}$ of a $m$-tensor field $f=(f_{i_1\cdots i_m})$ by the function $\mathcal{I}_{\alpha} f:L^2(S^m \tau_M ')\rightarrow L^2(\partial_+ \Omega M)$, where

\begin{align*}
[\mathcal{I}_{\alpha} f](x,\xi)=\int_{\tau_{-}(x,\xi)}^{0}\langle f(\gamma_{x,\xi}(\tau)),\dot{\gamma}_{x,\xi}^m (\tau)\rangle \exp\left(-\int_{\tau}^{0}\alpha(\gamma_{x,\xi}(\sigma),\dot{\gamma}_{x,\xi}(\sigma))\mathrm{d}\sigma\right) \mathrm{d}\tau.\\
\end{align*}
\end{dfn}

This definition can be extended to time-depending tensor fields. 

\begin{dfn}\label{D-DART}
For given $\alpha\in L^{\infty}(\Omega M)$ we define the $\textit{dynamic attenuated ray transform}$ of a $m$-tensor field $(f=f_{i_1\cdots i_m})$ by the function $\mathcal{I}_{\alpha}^d f:L^2(0,T; L^2(S^m \tau_M '))\rightarrow L^2(0,T; L^2(\partial_+ \Omega M))$ where

\begin{align*}
[\mathcal{I}_{\alpha}^d f](t, x,\xi)=\int_{\tau_{-}(x,\xi)}^{0}\langle f(t+\tau, \gamma_{x,\xi}(\tau)),\dot{\gamma}_{x,\xi}^m (\tau)\rangle \exp\left(-\int_{\tau}^{0}\alpha(\gamma_{x,\xi}(\sigma),\dot{\gamma}_{x,\xi}(\sigma))\mathrm{d}\sigma\right) \mathrm{d}\tau.\\
\end{align*}
\end{dfn}

In definitions \ref{D-ART} and \ref{D-DART} the function $\alpha$ acts as attenuation coefficient which is assumed to be known.

For further investigations it is necessary to introduce Bochner spaces $L^2(0,T; L^2(S^m \tau_M '))$ and $L^2(0,T; L^2(\partial_+ \Omega M))$ with norms

\begin{align}
||f||_{L^2(0,T; L^2(S^m \tau_M '))}&= \left(\int_{0}^{T}\int_M \langle f(\tau,x), f(\tau,x)\rangle \mathrm{d}V \mathrm{d}\tau\right)^{\frac{1}{2}}\\
||u||_{L^2(0,T; L^2(\partial_+ \Omega M))}&=\left(\int_{0}^{T} ||u(t)||_{L^2(\partial_+ \Omega M)}^2\mathrm{d}t\right)^{\frac{1}{2}}.
\end{align}

Analogously, Sobolev-Bochner spaces $H^k(0,T ; L^2(S^m \tau_M '))$ and $H^k (0,T; L^2(\partial_+ \Omega M))$ can be defined for all $k\in\mathbb{N}$. The ray transform on a CDRM can be continuously extended to

\begin{align*}
\mathcal{I}:H^k(0, T; H^k(S^m \tau_M '))\rightarrow H^k(0,T; H^k(\partial_+ \Omega M)).
\end{align*}

In \cite{Sharafutdinov_1994} it is proven that this linear operator is bounded if the tensor field is static. From this it can be easily concluded that the following applies to dynamic tensor fields and for any integer $k\geq 0$:

\begin{align*}
||\mathcal{I}f||_{H^k(0,T;H^k(\partial_+ \Omega M))}\leq ||f||_{H^k(0, T; H^k(S^m \tau_M '))}.
\end{align*}

The inverse problems that we focus on are to recover $f$ from given data $\mathcal{I}_{\alpha}^d f$, respectively $\mathcal{I}_{\alpha}^d f$, but not by inverting the integral transforms. Rather we consider inverse source problems for corresponding transport equations, which we will investigate further on.


\section{The ray transforms as solutions of transport equations}


\subsection{Derivation of the transport equation}


Given $\alpha\in L^{\infty}(\Omega M)$, $\alpha\geq 0$, and a $m$-tensor field $f=(f_{i_1\cdots i_m})\in L^2(0,T; L^2(S^m \tau_M '))$, we define the function $u:[0,T]\times T^0 M\rightarrow \mathbb{R}$ by

\begin{align}\label{solu}
u(t, x,\xi)=\int_{\tau_{-}(x,\xi)}^{0} f_{i_1,\dots, i_m}(t+\tau, \gamma_{x,\xi}(\tau))\dot{\gamma}_{x,\xi}^{i_1}\cdots\dot{\gamma}_{x,\xi}^{i_m} (\tau)\exp\left(-\int_{\tau}^{0}\alpha(\gamma_{x,\xi}(\sigma),\dot{\gamma}_{x,\xi}(\sigma))\mathrm{d}\sigma\right) \mathrm{d}\tau\\\nonumber
\end{align}

as an extension of $\mathcal{I}_{\alpha}^{d}f$ to $T^0 M$. We observe that for $(x,\xi)\in \partial_{-}\Omega M$ this integral vanishes whereas for $(x,\xi)\in \partial_{+}\Omega M$ it is identical to $\mathcal{I}_\alpha^d f$.

We show that \eqref{solu} is a solution of a transport equation. This is an extension of Sharafutdinov's result in \cite{Sharafutdinov_1994} to time-dependent fields with absorption. For constant refractive index $n$ a similar result is found in \cite{derevtsov}.

Let $(x,\xi)\in T^0 M\backslash T(\partial M)$ and $\gamma=\gamma_{x,\xi}:[\tau_{-}(x,\xi),\tau_{+}(x,\xi)]\rightarrow M$ be a geodesic defined by the initial conditions $\gamma_{x,\xi}(0)=x$, and $\dot{\gamma}_{x,\xi}(0)=\xi$. We choose a sufficiently small $s\in \mathbb{R}$ and put $t_s = t+s$, $x_{s}=\gamma(s)$ and $\xi_{s}=\dot{\gamma}(s).$  Then $\gamma_{x_{s},\xi_{s}}(\tau)=\gamma(\tau+s)$ and $\tau_{-}(x_{s},\xi_{s})=\tau_-(x,\xi)- s$ yielding 

\begin{align}
&u(t+s, x_{s},\xi_{s})\nonumber\\
&=\int_{\tau_{-}(x_{s},\xi_{s})}^{0} f_{i_1,\dots,i_m}(t_s+\tau,\gamma_{x_{s},\xi_{s}}(\tau))\dot{\gamma}_{x_{s},\xi_{s}}(\tau)^{i_1}\cdots\dot{\gamma}_{x_{s},\xi_{s}}(\tau)^{i_m}\exp\left(-\int_{\tau}^{0}\alpha(\gamma_{x_{s},\xi_{s}}(\sigma),\dot{\gamma}_{x_{s},\xi_{s}}(\sigma))\mathrm{d}\sigma\right)\mathrm{d}\tau\nonumber\\
&= \int_{\tau_{-}(x,\xi)}^{s}f_{i_1,\dots,i_m}(t+\tau, \gamma_{x,\xi}(\tau))\dot{\gamma}_{x,\xi}(\tau)^{i_1}\cdots\dot{\gamma}_{x,\xi}(\tau)^{i_m}\exp\left(-\int_{\tau}^{s}\alpha(\gamma_{x,\xi}(\sigma),\dot{\gamma}_{x,\xi}(\sigma))\mathrm{d}\sigma\right)\mathrm{d}\tau.\label{transport2}\\\nonumber
\end{align}

Next we differentiate this equation with respect to $s$ and evaluate it at $s=0$. We obtain for the left-hand side

\begin{align*}
\frac{\partial u}{\partial t} + \dot{\gamma}_k (0)\frac{\partial u}{\partial x_k} + \ddot{\gamma}_k(0)\frac{\partial u}{\partial \xi_k}
&=\frac{\partial u}{\partial t} + \dot{\gamma}_k(0)\frac{\partial u}{\partial x_k} -\Gamma_{ij}^{k}(\gamma(0))\dot{\gamma}_i(0)\dot{\gamma}_j(0)\frac{\partial u}{\partial \xi_k}\\
&=\frac{\partial u}{\partial t} + \langle \nabla_x u,\xi\rangle  - \Gamma_{ij}^{k}(x)\xi_i\xi_j\frac{\partial u}{\partial \xi_k}\\
&=\frac{\partial u}{\partial t} +  \mathcal{H}u,\\
\end{align*}

where $\mathcal{H}u := \langle \nabla_x u,\xi\rangle  - \Gamma_{ij}^{k}(x)\xi_i\xi_j\frac{\partial u}{\partial \xi_k}$ denotes the geodesic vector field. For brevity we define

\begin{align*}
U(\tau)&=f_{i_1,\dots,i_m}(t+\tau, \gamma_{x,\xi}(\tau))\dot{\gamma}_{x,\xi}^{i_1}(\tau)\cdots \dot{\gamma}_{x,\xi}^{i_m}(\tau),\\
V(\tau,s)&=\exp\left(-\int_{\tau}^{s}\alpha(\gamma_{x,\xi}(\sigma),\dot{\gamma}_{x,\xi}(\sigma))\mathrm{d}\sigma\right).\\
\end{align*}

Then the right hand side of \eqref{transport2} reads as

\begin{align*}
\int_{\tau_{-}(x,\xi)}^{s}U(\tau) V(\tau,s)\mathrm{d}\tau.\\
\end{align*}

Let us define $W(\tau,s)$ as an antiderivative of $U(\tau)V(\tau,s)$ with respect to $t$, i.e.

\begin{align*}
W(\tau,s)=\int U(\tau)V(\tau,s)\mathrm{d}\tau.\\
\end{align*}

Because of $\alpha\geq 0$ the function $V$ is bounded with $V(\tau,s)\leq 1$ and we obtain by the boundedness of $[\tau_{-}(x,\xi),0]$

\begin{align*}
\frac{\partial}{\partial s}W(\tau,s)=\int U(\tau)\frac{\partial V}{\partial s}(\tau,s)\mathrm{d}\tau.\\
\end{align*}

Hence,

\begin{align*}
\frac{\mathrm{d}}{\mathrm{d}s}\int_{\tau_{-}(x,\xi)}^{s}U(\tau) V(\tau,s)\mathrm{d}\tau
&= \frac{\mathrm{d}}{\mathrm{d}s}W(s, s)-\frac{\partial}{\partial s}W(\tau,s)\\
&=\frac{\partial W}{\partial \tau}(\tau,s)\Big\vert_{\tau=s} + \frac{\partial W}{\partial s}(\tau,s)\Big\vert_{\tau=s}-\frac{\partial}{\partial s}W(\tau,s)\\
&=U(s)\underbrace{V(s,s)}_{=1} +\int_{\tau_{-}(x,\xi)}^{s}U(\tau)\frac{\partial V}{\partial s}	(\tau,s)\mathrm{d}\tau\\
&=U(s) + \int_{\tau_{-}(x,\xi)}^{s}U(\tau)\frac{\partial V}{\partial s}(\tau,s)\mathrm{d}\tau .\\
\end{align*}

Using that

\begin{align*}
\frac{\partial V}{\partial s}(\tau,s)=-\alpha( \gamma_{x,\xi}(s),\dot{\gamma}_{x,\xi}(s)) V(\tau,s),\\
\end{align*}

we get 

\begin{align*}
\lim\limits_{s\rightarrow 0} \frac{\mathrm{d}}{\mathrm{d}s}\int_{\tau_{-}(x,\xi)}^{s}U(\tau) V(\tau,s)\mathrm{d\tau}
&= f_{i_1,\dots,i_m}(t, \gamma_{x,\xi}(0))\dot{\gamma}_{x,\xi}^{i_1}(0)\cdots\dot{\gamma}_{x,\xi}^{i_m}(0) - \alpha(x,\xi)\int_{\tau_{-}(x,\xi)}^{0}U(\tau)V(\tau,0)\mathrm{d}\tau \\
&= f_{i_1,\dots,i_m}(t, x)\xi^{i_1}\cdots\xi^{i_m} - \alpha(x,\xi)u(t, x,\xi).\\
\end{align*}

Finally, we arrive at

\begin{align}\label{trans_dyn}
\left(\frac{\partial }{\partial t}+\mathcal{H}+\alpha(x,\xi)\right)u(t,x,\xi)= f_{i_1,\dots,i_m}(t,x)\xi^{i_1}\cdots\xi^{i_m}.\\\nonumber
\end{align}

Note that furthermore $u$ satisfies the boundary conditions 

\begin{align}\label{bc_dyn}
u(t, x,\xi) = 
\left\{\begin{array}{ll}\mathcal{I}_\alpha^d f(t,x,\xi)=:\phi(t,x,\xi), & (x,\xi)\in \partial_+ \Omega M, t\in [0,T]\\
0, & (x,\xi)\in \partial_- \Omega M, t\in [0,T]\end{array}\right. .
\end{align}

In view of \eqref{solu} a natural initial value for $u$ is given by

\begin{align}\label{ic_dyn}
u(0,x,\xi) = 0
\end{align}

assuming that there is no flow $f$ for $t<0$.

For static tensor fields $f$ \eqref{trans_dyn} and \eqref{bc_dyn} turn into

\begin{align}\label{trans}
\left(\mathcal{H}+\alpha(x,\xi)\right)u(x,\xi)= f_{i_1,\dots,i_m}(x)\xi^{i_1}\cdots\xi^{i_m}.\\\nonumber
\end{align}

and 

\begin{align}\label{bc}
u(x,\xi) = \left\{\begin{array}{ll}\phi(x,\xi), & (x,\xi)\in \partial_+ \Omega M\\
0, & (x,\xi)\in \partial_- \Omega M\end{array}\right. \\\nonumber
\end{align}

for given $\phi= \mathcal{I}_\alpha f$. 

The inverse problems of computing $f$ from $\mathcal{I}_\alpha^d f$, $\mathcal{I}_\alpha f$, respectively, can now be re-formulated as inverse source problems for \eqref{trans_dyn}, \eqref{trans}: Compute $f$ from 
$\phi$ under the constraints \eqref{trans_dyn} and \eqref{bc_dyn}, \eqref{trans} and \eqref{bc}, respectively.
In this view it is very important that the parameter-to-solution map $f\mapsto u$ is well-defined what means that
the initial-, boundary-value problems have unique solutions. It turns out that indeed this is not satisfied. As a remedy we consider viscosity solutions. This is subject of the following section.


\subsection{Uniqueness of viscosity solutions for static tensor fields $f$}


We address existence and uniqueness of weak solutions for \eqref{trans_dyn} given the boundary and initial conditions \eqref{bc_dyn}, \eqref{ic_dyn}.

Let us first confine to static fields $f$. To derive the weak formulation of \eqref{trans} we multiply both sides by a test function $v\in H_0^1 (\Omega M)$ and integrate over $\Omega M$. Let $\hat{\phi}$ be a $H^1(\Omega M)$-extension, i.e. $\gamma_+ \hat{\phi}=\phi$, where 

\begin{align*}
\gamma_+ : H^1(\Omega M)\rightarrow L^2(\partial_+ \Omega M)
\end{align*}

denotes the trace operator restricting a function from $H^1(\Omega M)$ to $\partial_+ \Omega M$. Then the function $\tilde{u}= u - \hat{\phi}$ is in $H_0^1(\Omega M)$ and solves

\begin{align*}
(\mathcal{H}+\alpha)\tilde{u}=  f_{i_1,\dots,i_m}\xi^{i_1}\cdots\xi^{i_m} - (\mathcal{H}+\alpha)\hat{\phi}.
\end{align*}

This results in the weak formulation:

\fbox{\parbox{\columnwidth}{
Find $u_\phi = \tilde{u} + \hat{\phi}\in H^1 (\Omega M)$ such that

\begin{align*}
a(\tilde{u},v)=b_{\phi}(v),\qquad v \in H_0^1(\Omega M)
\end{align*}

where the bilinear form $a:H^1(\Omega M)\times H^1(\Omega M)\rightarrow \mathbb{R}$ is given as

\begin{align*}
a(u,v):= \int_{\Omega M} \left(\langle\nabla_x u,\xi\rangle  v - \Gamma_{ij}^{k}(x)\xi^i \xi^j \frac{\partial u}{\partial\xi^k}v +\alpha uv\right)\mathrm{d}\Sigma
\end{align*}

and the linear functional $b_\phi: H^1(\Omega M)\rightarrow \mathbb{R}$ as

\begin{align*}
b_\phi(v):= \int_{\Omega M} f_{i_1,\dots,i_m}\xi^{i_1}\cdots\xi^{i_m} v\mathrm{d}\Sigma - a(\hat{\phi},v).
\end{align*}
}}\\ \\

The bilinear form $a$ is not $H^1$-coercive, which is important to prove uniqueness of a weak solution according to standard results such as \cite[Theorem 2.1]{Lions2}. To overcome this difficulty we turn over to \textit{viscosity solutions} (c.f. \cite{Lions}). The idea of viscosity solutions is to transform the transport equation into an elliptic equation by adding a small multiple of the Laplace-Beltrami operator to the first-order differential operator of the original equation. For the arising elliptic problem we are able to prove unique solvability by using the Lax-Milgram Theorem.

The Laplace-Beltrami operator in $(M,g)$ can be computed as (c.f. \cite{minak1949}): 



\begin{align*}
\Delta &= \frac{1}{\sqrt{det g}}\sum_{i,j}\left(\frac{\partial}{\partial x_i} \left(\sqrt{det g} g^{ij} \frac{\partial }{\partial x_j}\right) + \frac{\partial}{\partial \xi_i} \left(\sqrt{det g} g^{ij} \frac{\partial }{\partial \xi_j}\right)\right)\\
&= n^{-3}(x)\sum_{i,j}\left(\frac{\partial}{\partial x_i}\left(n(x) \frac{\partial}{\partial x_i}\right) + n(x)\frac{\partial^2}{\partial \xi_i^2}\right)\\
&= n^{-2}(x)\sum_{i}\left(\frac{\partial^2}{\partial x_i^2}+\frac{\partial^2}{\partial \xi_i^2}\right)  + n^{-3}(x)\sum_i \frac{\partial n}{\partial x_i}\frac{\partial}{\partial x_i}.
\end{align*}

We split the operator into $\Delta = \Delta_x + \Delta_\xi$ where

\begin{align*}
\Delta_x := n^{-2}(x)\sum_{i}\frac{\partial^2}{\partial x_i^2}+n^{-3}(x)\sum_i \frac{\partial n}{\partial x_i}\frac{\partial}{\partial x_i},\qquad
\Delta_{\xi} := n^{-2}(x)\sum_{i}\frac{\partial^2}{\partial \xi_i^2}.
\end{align*}


Since $||\xi||=1$ we have

\begin{align*}
||\xi|| = \sqrt{g_{ij}\xi_i \xi_j}  =1 \Leftrightarrow ||\xi||_{eucl} = n^{-1}(x).\\
\end{align*}

Hence, $\xi$ reads in spherical coordinates as

\begin{align*}
\xi = (\xi_1, \xi_2, \xi_3)= n^{-1}(x) \begin{pmatrix} \cos\varphi\sin\theta\\ \sin\varphi\sin\theta \\ \cos\theta\end{pmatrix}.\\
\end{align*}

Simple calculations show

\begin{align}\label{xi1}
\frac{\partial }{\partial\xi_1}&=n(x)\left(\cos\varphi\cos\theta\frac{\partial }{\partial \theta} -\frac{\sin\varphi}{\sin\theta}\frac{\partial }{\partial \varphi}\right)\\
\frac{\partial }{\partial \xi_2} &= n(x)\left(\sin\varphi\cos\theta  \frac{\partial }{\partial \theta} + \frac{\cos\varphi}{\sin\theta} \frac{\partial }{\partial \varphi}\right)\label{xi2}\\
\frac{\partial }{\partial \xi_3} &= -n(x)\sin\theta \frac{\partial }{\partial \theta}.\label{xi3}\\\nonumber
\end{align}

%

The next step is to characterize the measure $\mathrm{d}\Sigma$ on $\Omega M$ by means of spherical coordinates. It holds that

\begin{align*}
\mathrm{d}\xi_1 &=n^{-1}(x)\left( -\sin\varphi\sin\theta \mathrm{d}\varphi + \cos\varphi\cos\theta \mathrm{d}\theta   \right)\\
\mathrm{d}\xi_2 &=n^{-1}(x)\left( \cos\varphi\sin\theta \mathrm{d}\varphi + \sin\varphi\cos\theta \mathrm{d}\theta   \right)\\
\mathrm{d}\xi_3 &=n^{-1}(x)\left( -\sin\theta  \mathrm{d}\theta   \right).\\
\end{align*}

Using the formula $3.6.33$ in \cite{Sharafutdinov_1994} we obtain

\begin{align}\label{surface_element}
\mathrm{d}\omega_x(\xi) &= g^{1/2}\sum_{i=1}^{3}(-1)^{i-1}\xi_i\mathrm{d}\xi_1 \wedge \cdots \wedge\widehat{\mathrm{d}\xi_i}\wedge\cdots\wedge \mathrm{d}\xi_3\nonumber\\
&= \sin\theta\mathrm{d}\theta\wedge \mathrm{d}\varphi.\\\nonumber
\end{align}

Thus,

\begin{align}\label{volume_element}
\mathrm{d}\Sigma= \mathrm{d}\omega_x(\xi)\wedge \mathrm{d}V^3 (x) = g^{1/2} \mathrm{d}\omega_x(\xi)\wedge \mathrm{d}x =n^{3}(x)\sin\theta\mathrm{d}\theta\wedge \mathrm{d}\varphi\wedge \mathrm{d}x.\\\nonumber
\end{align}

This will prove useful for later computations. The following two propositions are essential tools to prove the uniqueness of viscosity solutions.

\begin{proposition}\label{prop}
Let $u,v\in H^1(\Omega M)$. Then the following identity holds true:

\begin{align}\label{essential}
-\int_{\Omega_x M}\Gamma_{ij}^{k}\xi_i \xi_j \frac{\partial u}{\partial \xi_k} u \mathrm{d}\omega_x(\xi) = \int_{\Omega_x M}n^{-1}(x)\langle\nabla n,\xi\rangle  u^2 \mathrm{d}\omega_x (\xi).\\\nonumber
\end{align}

\end{proposition}
\begin{proof}
See Appendix \eqref{prop_proof}.
\end{proof}

\begin{proposition}
Let $u,v\in H^1(\Omega M)$. Then, we have

\begin{align}\label{Laplace_x}
-\int_{\Omega M} \Delta_x u v\mathrm{d}\Sigma &= \int_{\Omega M}\langle\nabla_x u,\nabla_x v\rangle \mathrm{d}\Sigma - \int_{\partial_+ \Omega M}v \nabla_{\nu} u\mathrm{d}\sigma_+\\
-\int_{\Omega M} \Delta_\xi u v\mathrm{d}\Sigma &= \int_{\Omega M}\langle\nabla_\xi u,\nabla_\xi v\rangle \mathrm{d}\Sigma.
\end{align}
\end{proposition}
\begin{proof}
The two statements follows directly from Green's formula and the fact that $\partial(\Omega_x M)=\emptyset$.
\end{proof}
%

\begin{corollary}
Let $u\in H_0^1(\Omega M)$. Then

\begin{align}\label{Laplace_x uu}
-\int_{\Omega M}\Delta_x u u\mathrm{d}\Sigma = \int_{\Omega M}\langle\nabla_x u, \nabla_x v\rangle \mathrm{d}\Sigma.
\end{align}
\end{corollary}
\begin{proof}
This is an immediate consequence of equation \eqref{Laplace_x}.
\end{proof}
%
%
%
%
%
%

Lastly, we need the next theorem for proving the uniqueness of solutions of general elliptic PDEs.

\begin{theorem}(Lax-Milgram Theorem \cite[Theorem 1.1]{Lions2})\newline
Let $V$ be a Hilbert space, $a(\cdot, \cdot):V\times V\rightarrow \mathbb{R}$ a coercive and continuous coercive bilinear form, i.e there exist $c_1, c_2>0$ such that

\begin{align*}
a(u,u)&\geq c_1 ||u||_V^2 \quad\forall u\in V\\
|a(u,v)| &\leq c_2 ||u||_V ||v||_V \quad\forall u, v\in V,\\
\end{align*}

and $b\in V'$ be a linear functional, i.e. there exist $c_3>0$ such that

\begin{align*}
|b(v)|\leq c_3 ||v||.
\end{align*}

Then the solution $u$ of the variational problem 
$$a(u,v)=b(v) \quad\forall v\in V$$
exists and is unique.
\end{theorem}

By definition, a viscosity solution to \eqref{trans} solves the equation

\begin{align}\label{viscoo}
-\varepsilon \Delta u + \langle \nabla_x u, \xi\rangle  +\alpha u - \Gamma_{ij}^{k}\xi_i \xi_j \frac{\partial u}{\partial \xi_k} = f_{i_1,\dots, i_m}(x)\xi^{i_1}\cdots \xi^{i_m},\\\nonumber
\end{align}

for $\varepsilon > 0$. Multiplying both sides with a test function $v\in H^1(\Omega M)$ and integrating over $\Omega M$ leads to

\begin{align*}
\int_{\Omega M} -\varepsilon\Delta u v + \langle \nabla_x u,\xi\rangle  v +\alpha uv - \Gamma_{ij}^{k}\xi_i \xi_j \frac{\partial u}{\partial \xi_k}v \mathrm{d}\Sigma = \int_{\Omega M} f_{i_1,\dots, i_m}(x)\xi^{i_1}\cdots \xi^{i_m} v\mathrm{d}\Sigma.\\
\end{align*}

We derive the variational formulation of the boundary value problem by setting

\begin{align}
a_{\varepsilon}(u,v)&= \int_{\Omega M} -\varepsilon\Delta u v  \mathrm{d}\Sigma + a(u,v)\nonumber\\
&=\int_{\Omega M} \varepsilon \langle\nabla u, \nabla v\rangle   \mathrm{d}\Sigma- \int_{\partial_+ \Omega M}v \nabla_{\nu} u\mathrm{d}\sigma_+  + a(u,v)\label{auv}\\
b_\phi^{\varepsilon}(v)&= \int_{\Omega M} f_{i_1,\dots, i_m}(x)\xi^{i_1}\cdots \xi^{i_m} v\mathrm{d}\Sigma - a_\varepsilon (\hat{\phi},v).\label{bv}\\\nonumber
\end{align}

Consequently the weak form of \eqref{viscoo} along with the boundary condition \eqref{bc} is given by

\fbox{\parbox{\columnwidth}{
Find $u_{\phi,\varepsilon}=u_\varepsilon + \hat{\phi}\in H^1(\Omega M)$ such that

\begin{align}\label{variation}
a_{\varepsilon}(u_\varepsilon,v)=b_\phi^{\varepsilon}(v),\qquad \forall v\in H_{0}^{1}(\Omega M),\\\nonumber
\end{align}

where $u_\varepsilon\in H_0^1(\Omega M)$ and $\gamma_+ \hat{\phi}=\phi$.

}} \\ \\

The variational problem \eqref{variation} has in fact a unique solution.

\begin{theorem}\label{bigtheorem}
Let $\varepsilon >0$, $\alpha\in L^\infty (\Omega M)$ with $\alpha(x,\xi)\geq \alpha_0 >0$ for all $(x,\xi)\in \Omega M$, $n\in C^1(M)$ and $f\in L^2(S^m \tau_M ')$ a $m$-tensor field.
If 

\begin{align}\label{vierus_cond}
 \sup_{x\in M} \frac{||\nabla n(x)|| }{n(x)} < \alpha_0,\\\nonumber
\end{align}

then the solution $u_\varepsilon\in H^1(\Omega M)$ of the variational problem \eqref{variation} exists and is unique.
\end{theorem}

\begin{proof}
The proof consists of an application of the Lax-Milgram Theorem. To this end we have to show

\begin{itemize}
\item the coercivity of $a_\varepsilon$,
\item the continuity of $a_{\varepsilon}$ and
\item the continuity of $b^\varepsilon_\phi$.
\end{itemize}
Let $0<\delta < 1$ be sufficiently small such that

\begin{align}
 \sup_{x\in M} \frac{||\nabla n(x)|| }{n(x)} < (1-\delta)\alpha_0\\\nonumber
\end{align}

is satisfied. Since $v=0$ on $\partial \Omega M$ the boundary integral in \eqref{auv} vanishes. We split $a_{\varepsilon}=a_{\varepsilon}^{(1)} + a_{\varepsilon}^{(2)}$, where

\begin{align*}
a_{\varepsilon}^{(1)}(u,v) &= \int_{\Omega M} \varepsilon\langle\nabla_x u, \nabla_x v\rangle  + \langle\nabla_x u,\xi\rangle  v + \delta\alpha u v\mathrm{d}\Sigma\\ \\
a_{\varepsilon}^{(2)}(u,v) &= \int_{\Omega M} \varepsilon \langle\nabla_\xi u, \nabla_\xi v\rangle  -\Gamma_{jk}^{i}(x)\xi^j\xi^k\frac{\partial u}{\partial \xi^i} v + (1-\delta)\alpha u v\mathrm{d}\Sigma.\\
\end{align*}

One verifies that

\begin{align*}
\int_{\Omega M}\langle \nabla_x u,\xi\rangle  u\mathrm{d}\Sigma&= \int_{\partial_+ \Omega M}u^2 \langle\xi, \nu\rangle \mathrm{d}\sigma - \int_{\Omega M}\langle\nabla_x u,\xi\rangle  u\mathrm{d}\Sigma,\\
\end{align*}

where $\mathrm{d}\sigma$ is the measure on $\partial\Omega M$. Hence,

\begin{align*}
\int_{\Omega M}\langle \nabla_x u,\xi\rangle  u\mathrm{d}\Sigma
&=\frac{1}{2} \int_{\partial_+ \Omega M}\phi^2 \langle\xi, \nu\rangle \mathrm{d}\sigma\geq 0\\
\end{align*}

and, consequently,

\begin{align*}
a_{\varepsilon}^{(1)}(u,u)\geq \int_{\Omega M} \varepsilon||\nabla_x u|| ^2 + \delta u^2 \mathrm{d}\Sigma.\\
\end{align*}

Using equation \eqref{essential} and condition \eqref{vierus_cond} we estimate the second part by

\begin{align*}
a_{\varepsilon}^{(2)}(u,u) &= \int_{\Omega M} \varepsilon||\nabla_\xi u|| ^2 + \left((1-\delta)\alpha + n^{-1}(x)\langle \nabla n(x),\xi\rangle  \right) u^2\mathrm{d}\Sigma\\
&\geq \int_{\Omega M} \varepsilon||\nabla_\xi u|| ^2+ \left((1-\delta)\alpha - n^{-1}||\nabla n(x)|| \right) u^2 \mathrm{d}\Sigma\\
&\geq \int_{\Omega M} \varepsilon||\nabla_\xi u|| ^2 \mathrm{d}\Sigma.\\
\end{align*}

Adding both parts we have the coercivity condition

\begin{align}
a_{\varepsilon}(u,u) &\geq \int_{\Omega M} \varepsilon(||\nabla_x u|| ^2 + ||\nabla_\xi u|| ^2) + \delta u^2\mathrm{d}\Sigma\nonumber\\
&\geq \min\left(\varepsilon, \delta\right) ||u||_{H^1(\Omega M)}^2.\\\nonumber
\end{align}

Next we prove the continuity of $a$. Using the triangle inequality and \eqref{auv} gives

\begin{align}
|a_{\varepsilon}(u,v)|&\leq  \Big\vert\int_{\Omega M}\varepsilon\left(\langle\nabla_x u, \nabla_x v\rangle  + \langle\nabla_\xi u, \nabla_\xi v\rangle \right)\mathrm{d}\Sigma\Big\vert + \Big\vert\int_{\Omega M}\xi_k \frac{\partial u}{\partial x_k} v\mathrm{d}\Sigma\Big\vert \nonumber\\
&+\Big\vert\int_{\Omega M}\alpha uv \mathrm{d}\Sigma\Big\vert +\Big\vert\int_{\Omega M} \Gamma_{ij}^{k}\xi_i \xi_j \frac{\partial u}{\partial \xi_k}v \mathrm{d}\Sigma\Big\vert. \label{a_cont} \\
\end{align}

The first summand can be estimated by using the Cauchy-Schwarz inequality

\begin{align}
\Big\vert\int_{\Omega M} \varepsilon \left(\langle\nabla_x u, \nabla_x v\rangle  + \langle\nabla_\xi u, \nabla_\xi v\rangle \right)\mathrm{d}\Sigma\Big\vert
&\leq \varepsilon ||u||_{H^1(\Omega M)} ||v||_{H^1(\Omega M)}.\label{laplace_cont}\\\nonumber
\end{align}

In the same manner we obtain for the second summand

\begin{align}
\Big\vert\int_{\Omega M}\langle\nabla_x,\xi\rangle  v\mathrm{d}\Sigma\Big\vert
&=\Big\vert\int_{\Omega M}\langle \nabla_x u, v \xi\rangle  \mathrm{d}\Sigma\Big\vert\nonumber\\
&\leq \left( \int_{\Omega M}\langle\nabla_x u, \nabla_x u\rangle  \mathrm{d}\Sigma \right)^ {\frac{1}{2}}\cdot \left(\int_{\Omega M}v^2 \mathrm{d}\Sigma \right)^{\frac{1}{2}}\nonumber\\
&\leq ||u||_{H^1(\Omega M)}||v||_{H^1(\Omega M)}.\label{kp_cont}\\\nonumber
\end{align}

The absorption term can be estimated by

\begin{align}
\Big\vert\int_{\Omega M} \alpha u v\mathrm{d}\Sigma \leq ||\alpha||_{L^\infty (\Omega M)} ||u||_{L^2(\Omega M)} ||v||_{L^2(\Omega M)}\leq ||\alpha||_{L^\infty (\Omega M)} ||u||_{H^1(\Omega M)} ||v||_{H^1(\Omega M)}.\label{att_cont}\\\nonumber
\end{align}

For the last part in \eqref{a_cont} we use \eqref{christ_refractive} and obtain

\begin{align}
\Big\vert\int_{\Omega M} -\Gamma_{ij}^{k}(x)\xi^i\xi^j \frac{\partial u}{\partial \xi_k} v\mathrm{d}\Sigma\Big\vert
&= \Big\vert\int_{\Omega M} n^{-3}(x)\left(\frac{\partial n}{\partial x_k} - 2\xi_k\langle\nabla n,\xi\rangle  \right)\frac{\partial u}{\partial \xi_k}v \mathrm{d}\Sigma\Big\vert\nonumber\\
&\leq \int_{\Omega M} ||\nabla n(x)|| (n^{-1}(x) + 2n^{-3}(x))||\nabla_\xi u||  |v| \mathrm{d}\Sigma\nonumber\\
&\leq \int_{\Omega M}3 \frac{||\nabla n(x)|| }{n(x)} ||\nabla_\xi u||  |v| \mathrm{d}\Sigma\nonumber\\
&\leq 3||\alpha||_{L^{\infty}(\Omega M)}||\nabla_\xi u||_{L^2(\Omega M)}||v||_{L^2(\Omega M)}\nonumber\\
&\leq 3||\alpha||_{L^{\infty}(\Omega M)}|| u||_{H^1(\Omega M)} ||v||_{H^1(\Omega M)}.\label{last_cont}\\\nonumber
\end{align}

Finally, with \eqref{laplace_cont} - \eqref{last_cont} we arrive at

\begin{align*}
|a_{\varepsilon}(u,v)|\leq (\varepsilon + 1 +  4||\alpha ||_{L^\infty(\Omega M)}) ||u||_{H^1(\Omega M)} ||v||_{H^1(\Omega M)}.\\
\end{align*}

The last step is to prove continuity of $b^\varepsilon_\phi$. We compute

\begin{align*}
\Big\vert \int_{\Omega M} f_{i_1,\dots i_m}(x)\xi^{i_1}\cdots \xi^{i_m} v \mathrm{d}\Sigma\Big\vert
&\leq c(n)\left(   \int_{\Omega M}\langle f(x),\xi^m\rangle 
^2\mathrm{d}\Sigma\right)^{\frac{1}{2}}\cdot \left(   \int_{\Omega M}v^2\mathrm{d}\Sigma\right)^{\frac{1}{2}}\\
&\leq c(n)||f||_{L^2(S^m \tau_M ')}||v||_{H^1(\Omega M)}\\
\end{align*}

for a positive constant $c(n)$ depending on $n$. The continuity of $b^\varepsilon_\phi$ then follows from this estimate and the continuity of $a_\varepsilon$. This completes the proof.

\end{proof}

\begin{remark}
The continuity conditions for $a_\varepsilon$ and $b_\phi^\varepsilon$ hold true also for $\varepsilon=0$, whereas the coercivity only holds for $\varepsilon>0$. Theorem \ref{bigtheorem} guarantees that there exists a unique, weak viscous solution if $n$ varies only slowly. Especially in the Euclidean geometry ($n=1$) condition \eqref{vierus_cond} is valid for any positive $\alpha_0$. This is in accordance with the results in \cite{derevtsov}.
\end{remark}


\subsection{Extension to time-depending tensor fields $f$}


Let $V$ be a reflexive and separable Banach space with norm $||\cdot||_V$ and $V^*$ its dual space with norm $||\cdot||_{V^*}$. The dual pairing is denoted by $\langle \cdot, \cdot\rangle_{V\times V^*}$ . We define the Lebesgue-Bochner space $L^2(0,T;V)$ as the space of all $V$-valued functions $u$ on $(0,T)$ for which $t\mapsto ||u(t)||_{V}$ is a function in $L^2([0,T])$. Equipped with the norm

\begin{align*}
||u||_{L^2(0,T;V)}:= \left(\int_{0}^{T}||u(t)||_{V}^2 \mathrm{d}t\right)^{\frac{1}{2}}
\end{align*}

$L^2(0,T;V)$ turns into a Banach space. Moreover, let 

\begin{align}
W^{1,1,2}(V, V^*) = \lbrace u\in L^2(0,T; V): \mathrm{d}_t u\in L^2(0,T;V^*)\rbrace,\\ \nonumber
\end{align}

where $\mathrm{d}_t u$ is the distributional derivative of $u$. In the following we always consider the case that $V = H_0^1(\Omega M)$ leading to $V^*  = H^{-1}(\Omega M)$. 
We interpret \eqref{variation} as an abstract operator equation (c.f. \cite{showalter1997},\cite{minity1963}) in the sense that

\begin{align*}
A_{\varepsilon}(u) = f_{i_1,\dots,i_m}\xi^{i_1}\cdots\xi^{i_m}\quad \textrm{ in }V^*,\\
\end{align*}

where $A_{\varepsilon}:V\rightarrow V^*$ defined by $A_{\varepsilon}u = a_{\varepsilon}(u,\cdot)$ is a monotone operator. This result can applies also to the dynamic equation following \cite{zeid} 
and \cite{roubicek}. As seen in \eqref{trans} $u=\mathcal I_\alpha^d f$ satisfies

\begin{align}
\frac{\partial u}{\partial t} + (\mathcal{H} + \alpha) u = f_{i_1,\dots,i_m}(t,x)\xi^{i_1}\cdots\xi^{i_m}.\\\nonumber
\end{align}

The corresponding viscosity solution is characterized by

\begin{align}
\frac{\partial u}{\partial t} - \varepsilon\Delta u +(\mathcal{H} + \alpha) u = f_{i_1,\dots,i_m}(t,x)\xi^{i_1}\cdots\xi^{i_m}.\\\nonumber
\end{align}

%

The associated variational formulation reads as:

\fbox{\parbox{\columnwidth}{
Find $u_\varepsilon\in W^{1,1,2}(0,T; V,V^{*})$ such that

\begin{align}\label{dyn_pde}
\langle \mathrm{d}_t u_\varepsilon(t),v\rangle_{V^*, V} + a_{\varepsilon}(t; u_\varepsilon(t), v)&=\langle b_\phi^\varepsilon(t),v\rangle_{V^*, V},\\
u_\varepsilon(0)&=0,\nonumber\\\nonumber
\end{align}

 for all $v\in V$ and for a.e. $t\in (0,T)$ and set $u_{\phi,\varepsilon}^d=u_\varepsilon + \hat{\phi}$.

}} \\ \\

The bilinear form $a_{\varepsilon}$ is defined similar to \eqref{auv} by

\begin{align*}
a_{\varepsilon}(t;u,v)=\int_{\Omega M} \varepsilon\langle\nabla u(t),\nabla v(t)\rangle  + \langle \nabla_x u(t),\xi\rangle v(t) - \Gamma_{ij}^{k}\xi_i \xi_j \frac{\partial u(t)}{\partial \xi_k}v(t) +\alpha u(t)v(t) \mathrm{d}\Sigma - \int_{\partial_+ \Omega M}\varepsilon\nabla_\nu u(t)v(t)\mathrm{d}\sigma_+\\
\end{align*}

and the linear form $b_\phi^\varepsilon$ is given by

\begin{align*}
\langle b_\phi^\varepsilon(t),v\rangle = \int_{\Omega M} f_{i_1,\dots,i_m}(t,x)\xi^{i_1}\cdots\xi^{i_m} v(t)\mathrm{d}\Sigma - a_\varepsilon(t,\hat{\phi}(t),v(t)).\\
\end{align*}

Note that, since $W^{1,1,2}(V, V^*) \subset \mathcal{C}(0,T;L^2 (\Omega M))$ by the Aubin-Lions Lemma, the point evaluation $u_\varepsilon (0)$ in \eqref{dyn_pde} is well-defined. The next theorem is a typical tool that is used to guarantee unique solutions of time-dependent differential equations. $ $\newline

%
%
%
%

\begin{theorem}[Theorem 3.6 in \cite{alphonse_2015}]\label{uniq_dynamic}
Let $V$ be a reflexive Banach space. Assume $b\in V^{*}$ and that the bilinear form $a(t;\cdot,\cdot):V\times V\rightarrow \mathbb{R}$ satisfies the following properties:
\begin{itemize}
\item The mapping $t\mapsto a(t;u,v)$ is measurable for all $u,v\in V$.
\item There exist a $c_1>0$: $a(t,u,v)\geq c_1 ||u||_V^2$  for all $t\in(0,T)$.
\item There exist a $c_2>0$: $|a(t,u,v)|\leq c_2 ||u||_V ||v||_V$ for all $t\in(0,T)$.
\end{itemize}

Then the equation

\begin{align*}
\langle\mathrm{d}_t u(t),v\rangle_{V^* , V}+a(t,u(t),v)=\langle b, v\rangle_{V^*, V}\qquad\forall v\in V \\
\end{align*}

has a unique solution $u\in W^{1,1,2}(0,T;V, V^*)$ satisfying

\begin{align}
||u||_{W^{1,1,2}(V, V^*)}\leq \frac{1}{c_1}||b||_{V^*}.\\ \nonumber
\end{align}
\end{theorem}

Using theorem \ref{uniq_dynamic} we immediately one of the main results of this paper.

\begin{theorem}\label{T_main_theorem_dyn}
Let $\varepsilon >0$, $\alpha\in L^\infty (\Omega M)$ with $\alpha(x,\xi)\geq \alpha_0 >0$ for all $(x,\xi)\in \Omega M$, $n\in C^{1}(M)$ and $f\in L^2(0,T; S^m \tau_M ')$ a $m$-tensor field. If 

\begin{align}\label{vierus_cond_2}
\sup_{x\in M} \frac{||\nabla n(x)|| }{n(x)} < \alpha_0,\\\nonumber
\end{align}

then the variational problem \eqref{dyn_pde} has a unique solution $u_\varepsilon$.
\end{theorem}

\begin{proof}
The assumption follows directly from \eqref{bigtheorem}, \eqref{uniq_dynamic} and the fact that the bilinear form $a_{\varepsilon}$ is continuous and hence measurable. 
\end{proof}

Summarizing theorems \ref{bigtheorem} and \ref{T_main_theorem_dyn} static and dynamic tensor field tomography in a medium with absorption and refraction can be mathematically modeled by the linear equations

\[    \mathcal{F}_\alpha (f) = \phi,\qquad\qquad \mathcal{F}_\alpha^d (f) = \phi    \]

for given data $\phi$, where $\mathcal{F}_\alpha : L^2 (S^m \tau_M ')\to L^2 (\partial_+ \Omega M)$, 
$\mathcal{F}_\alpha^d: W^{1,1,2}(0,T; V,V^{*})\to L^2(0,T;\partial_+ \Omega M)$ can be decomposed as $\mathcal{F}_\alpha = \gamma_+ \circ \mathcal{S}_\alpha$, $\mathcal{F}_\alpha^d = \gamma_+ \circ \mathcal{S}_\alpha^d$ with parameter-to-solution mappings

\begin{align*}
&\mathcal{S}_\alpha : L^2 (S^m \tau_M ') \to L^2 (\partial \Omega M),\qquad f \mapsto u_{\phi,\varepsilon},\\
&\mathcal{S}_\alpha^d : W^{1,1,2}(0,T; V,V^{*}) \to L^2(0,T;\partial \Omega M),\qquad f \mapsto u_{\phi,\varepsilon}^d.
\end{align*}

Theorems \ref{bigtheorem} and \ref{T_main_theorem_dyn} then guarantee that all mappings are well-defined.

%
%
%

%


\section{Numerical validation for $\varepsilon\to 0$}

It is still an open question whether $\lim_{\varepsilon\to 0} u_{\phi,\varepsilon}^{(d)}$ exists (and in which topology) and solves the original
transport equations \eqref{trans}, \eqref{trans_dyn}, respectively. This is very important also regarding the corresponding inverse source problems. At least we are able to prove numerical evidence by the following example:\\[2ex]

Let $M$ be the 2D unit ball and $f:M\rightarrow\mathbb{R}^2$ a vector field on $M$ defined by

\begin{align*}
f(x_1, x_2)= \begin{pmatrix}
\frac{1}{x_1^2 +x_2^2 +1}\\ x_1 +x_2
\end{pmatrix}.
\end{align*}

We choose $n(x)=x_1^2 + x_2^2 + 1.5$ and $\alpha = 1$ such that \eqref{vierus_cond} is satisfied. Now, consider the discretization $$u_{ijk}=u(x_{ij},\xi_{ijk})$$ of \eqref{trans} and \eqref{viscoo} where $x_{ij}= r_i (\cos\phi_j, \sin\phi_j)$, $\xi_{ijk} = n^{-1}(x_{ij})(\cos\theta_k,\sin\theta_k)$ and 

\begin{align}
r_i=\frac{i}{I},\quad i=1,\dots,I\qquad
\phi_j = \frac{2\pi j}{J},\quad j=1,\dots,J\qquad
\theta_k = \frac{2\pi k}{K},\quad k=1,\dots,K.
\end{align}

\begin{figure}[!h]
\begin{center} 
   \includegraphics[scale=0.37]{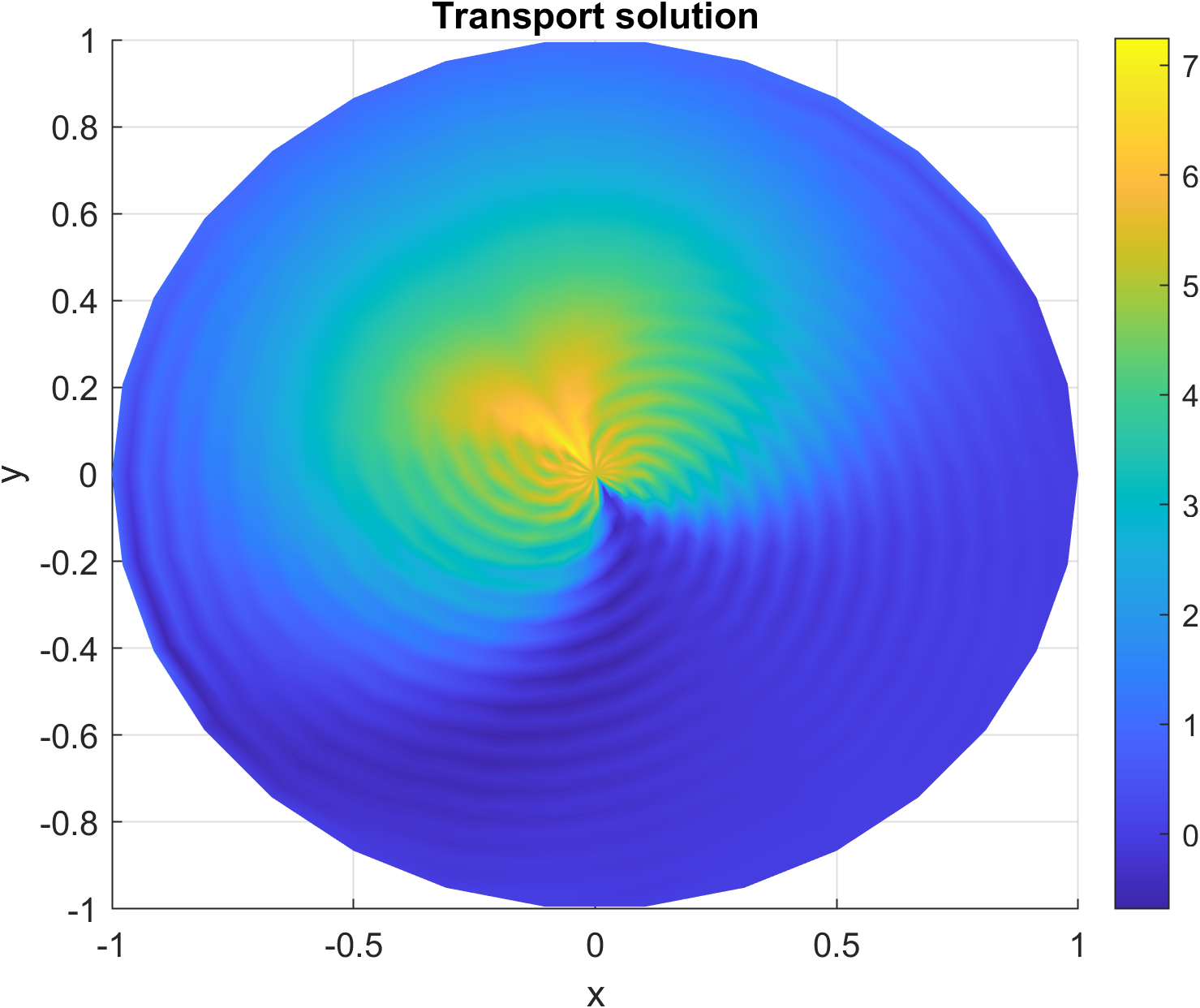}\qquad 
   \includegraphics[scale=0.37]{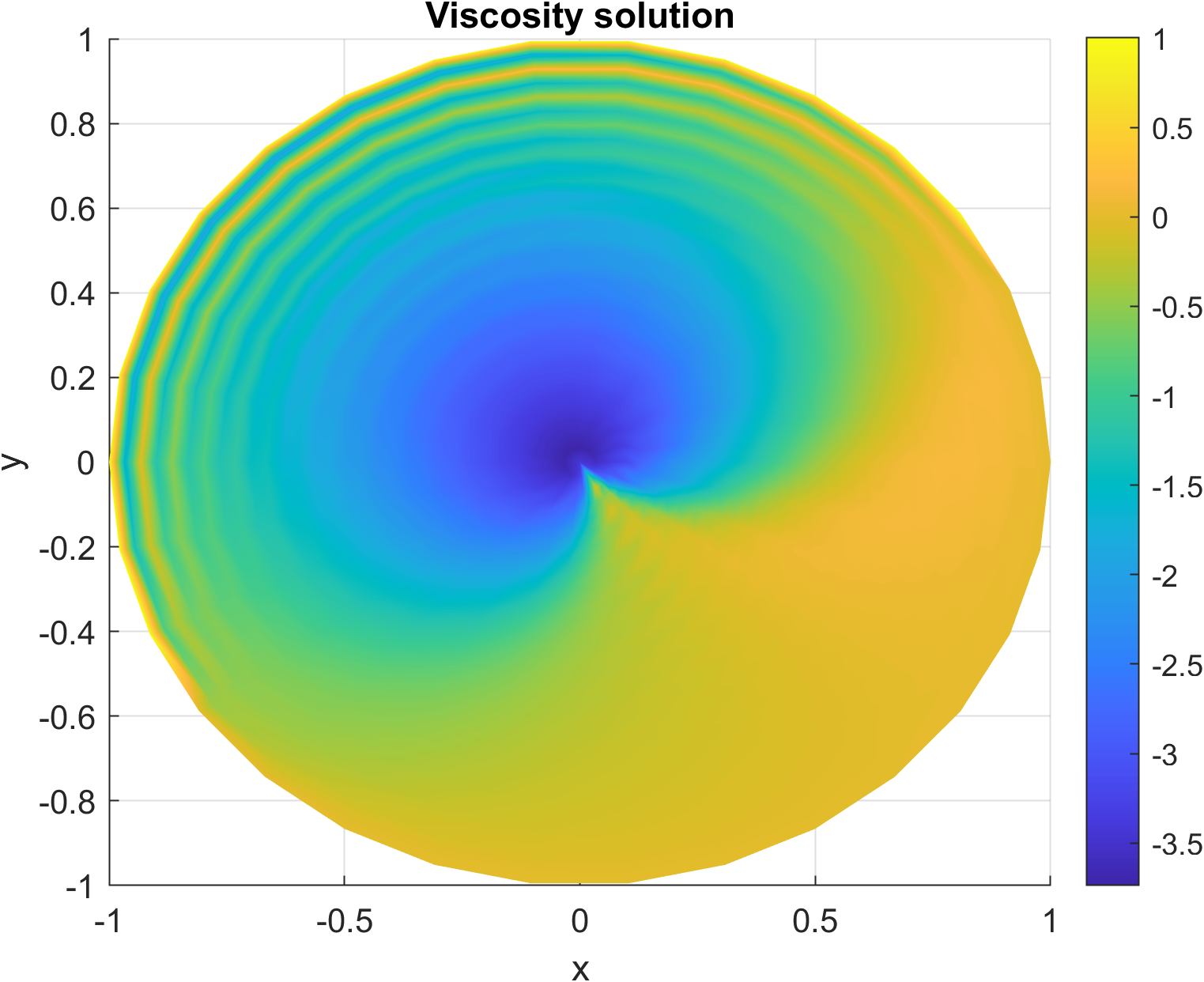}\qquad
   \includegraphics[scale=0.37]{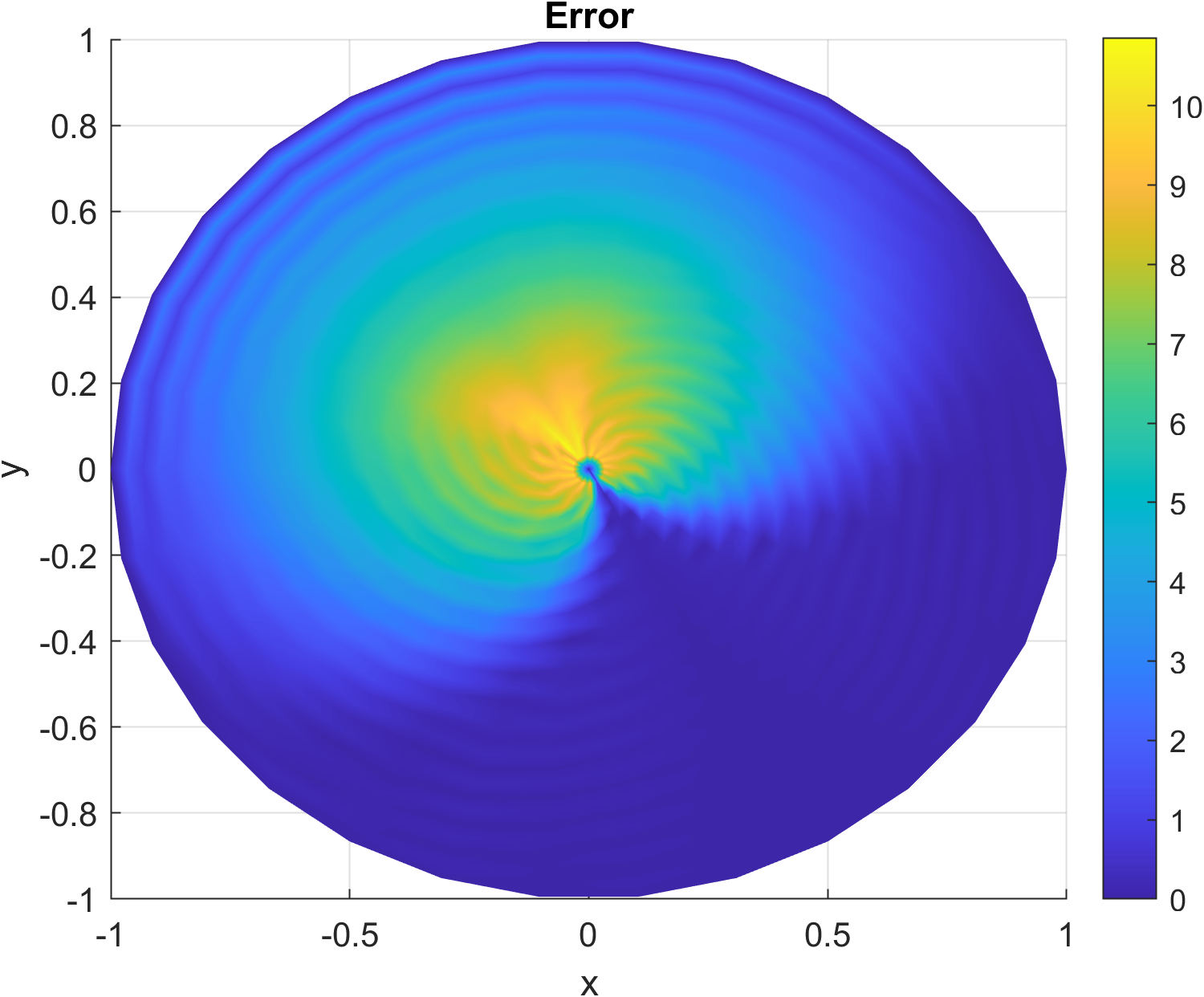}
  \caption{Solution of transport equation, viscosity equation and the relative error for $(I,J,K)=(30,30,10)$ and $\varepsilon=10^{-3}$}
  \end{center}
\end{figure}

\begin{figure}[!h]
\begin{center} 
   \includegraphics[scale=0.37]{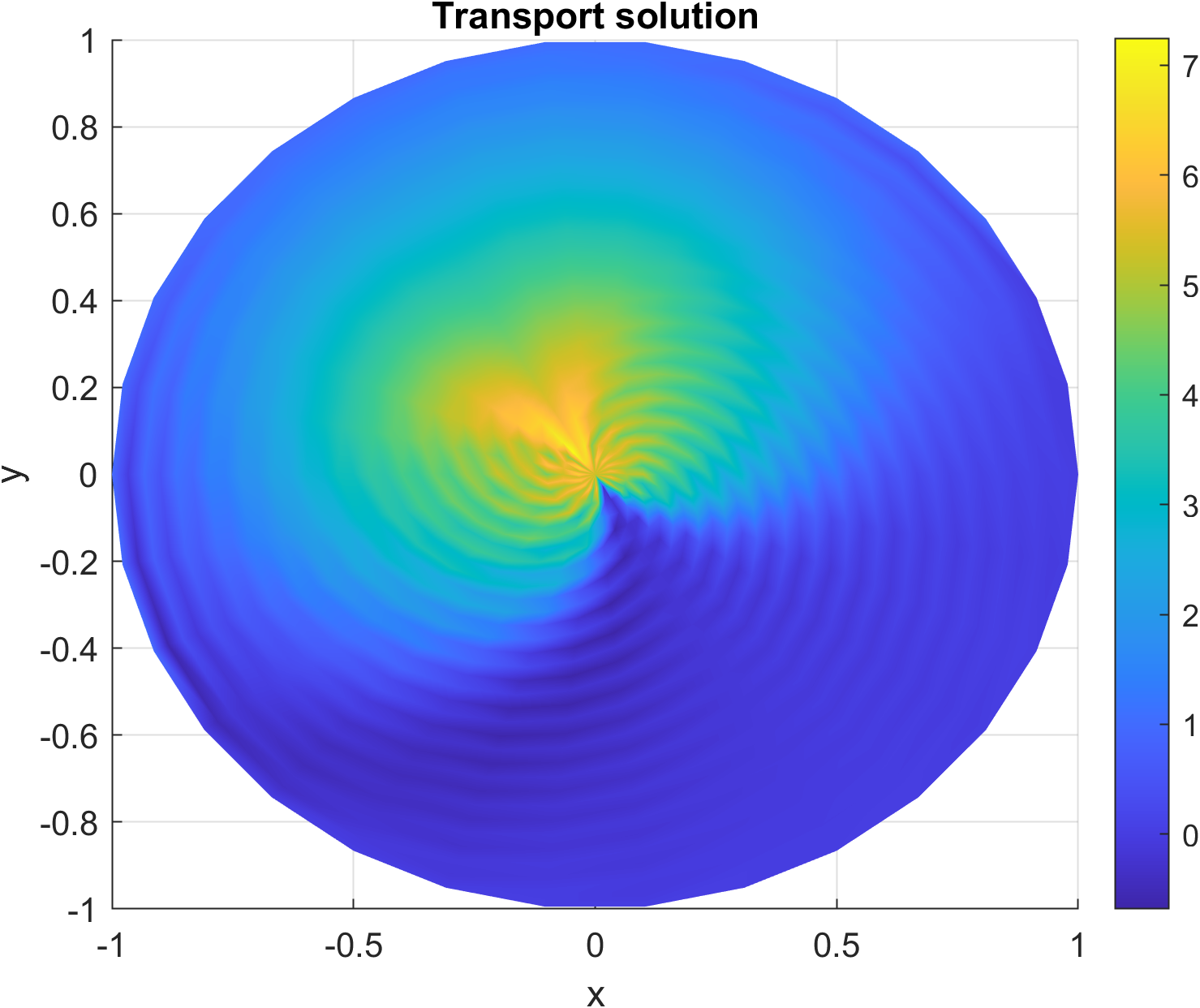}\qquad 
   \includegraphics[scale=0.37]{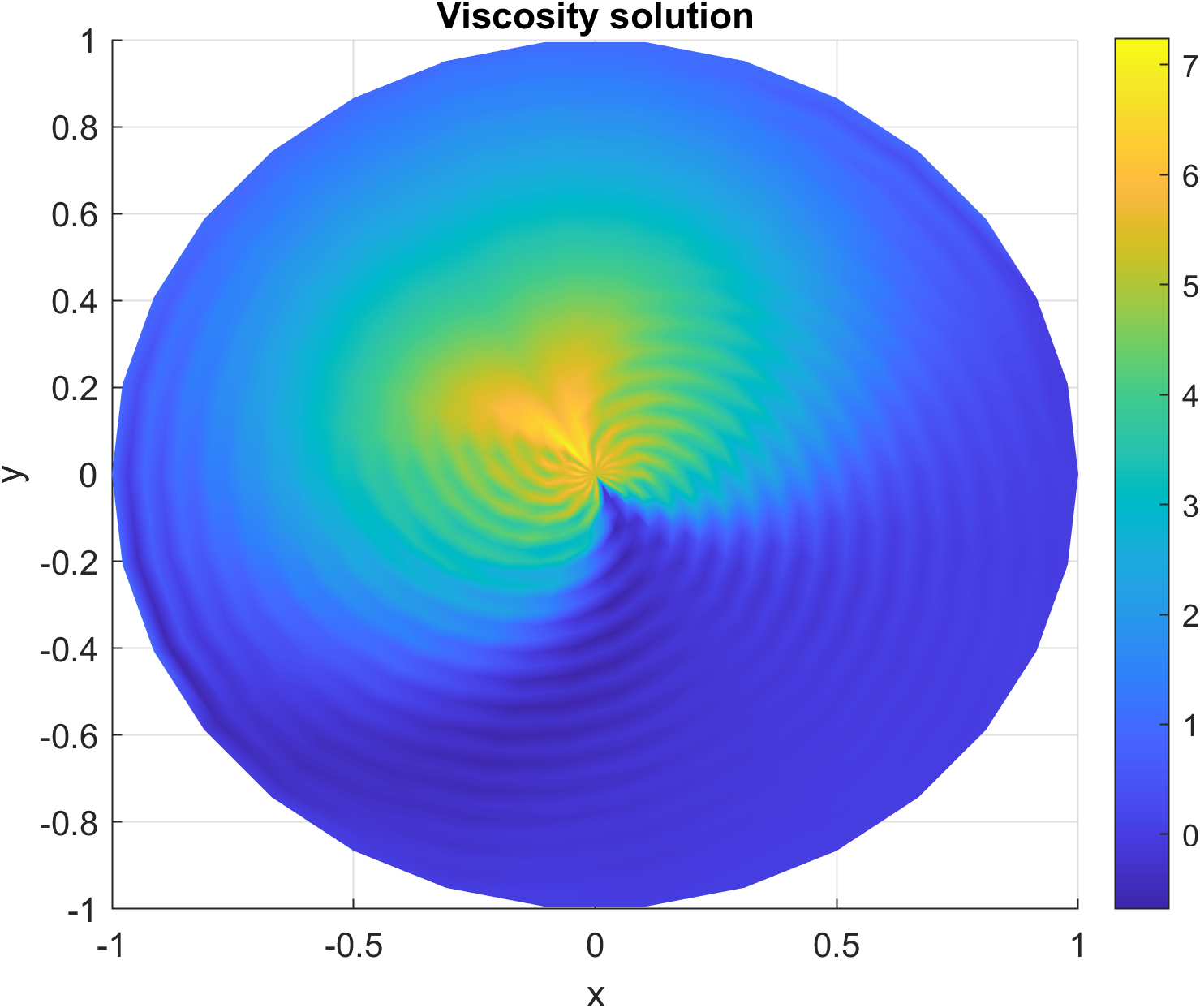}\qquad
   \includegraphics[scale=0.37]{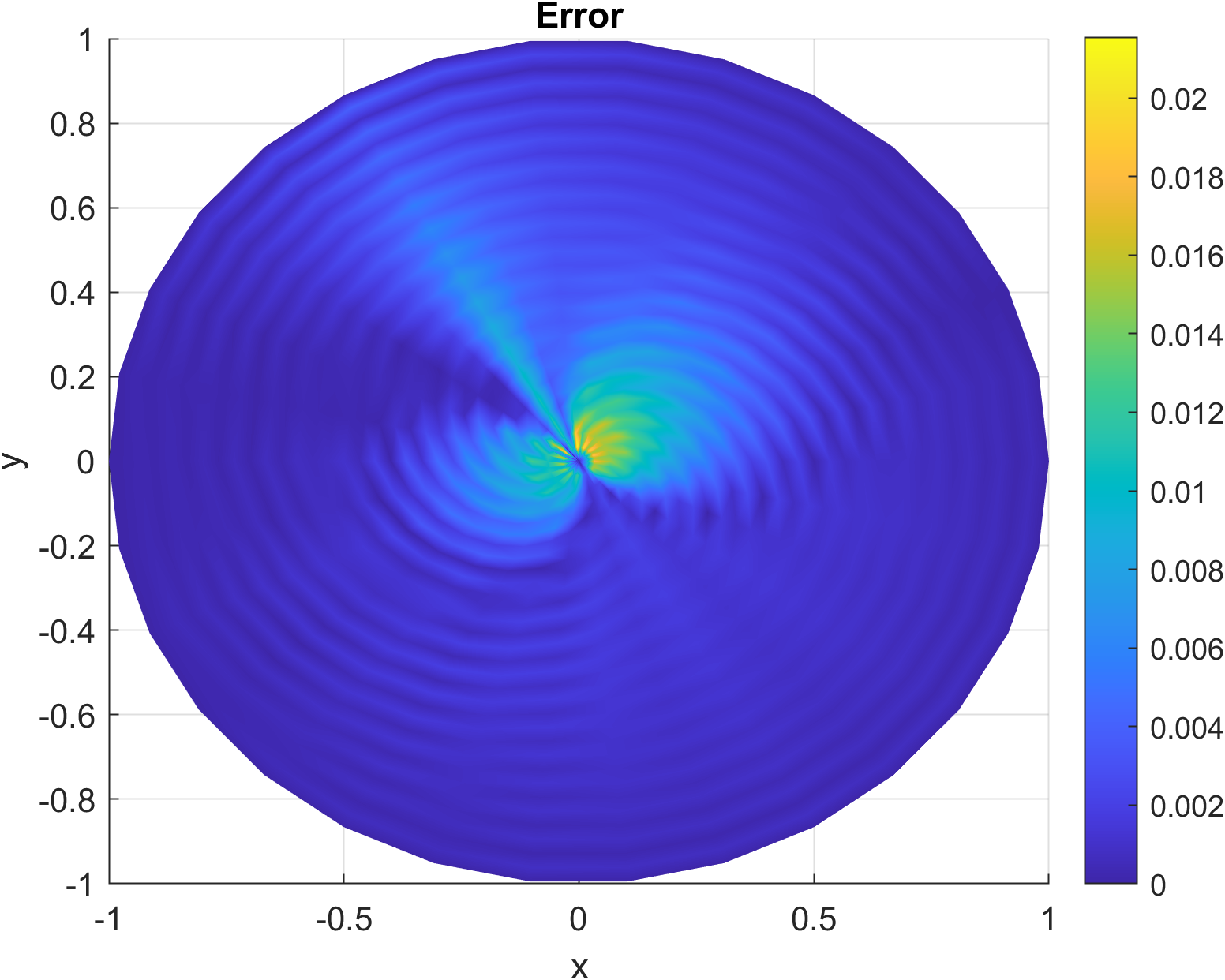}
  \caption{Solution of transport equation, viscosity equation and the relative error for $(I,J,K)=(30,30,10)$ and $\varepsilon=10^{-6}$}
  \end{center}
\end{figure}  
  
  \begin{figure}[!h]
\begin{center} 
   \includegraphics[scale=0.37]{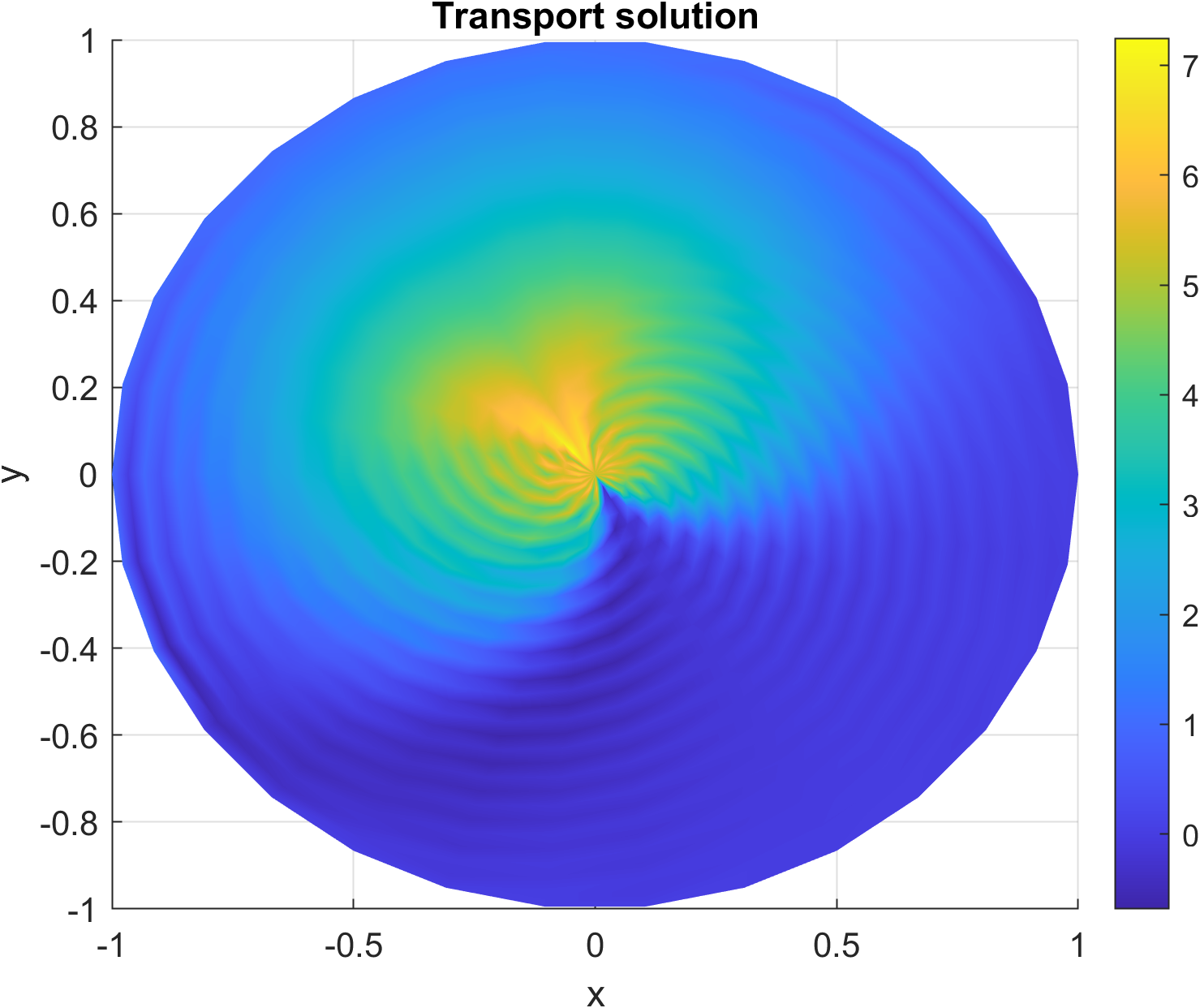}\qquad 
   \includegraphics[scale=0.37]{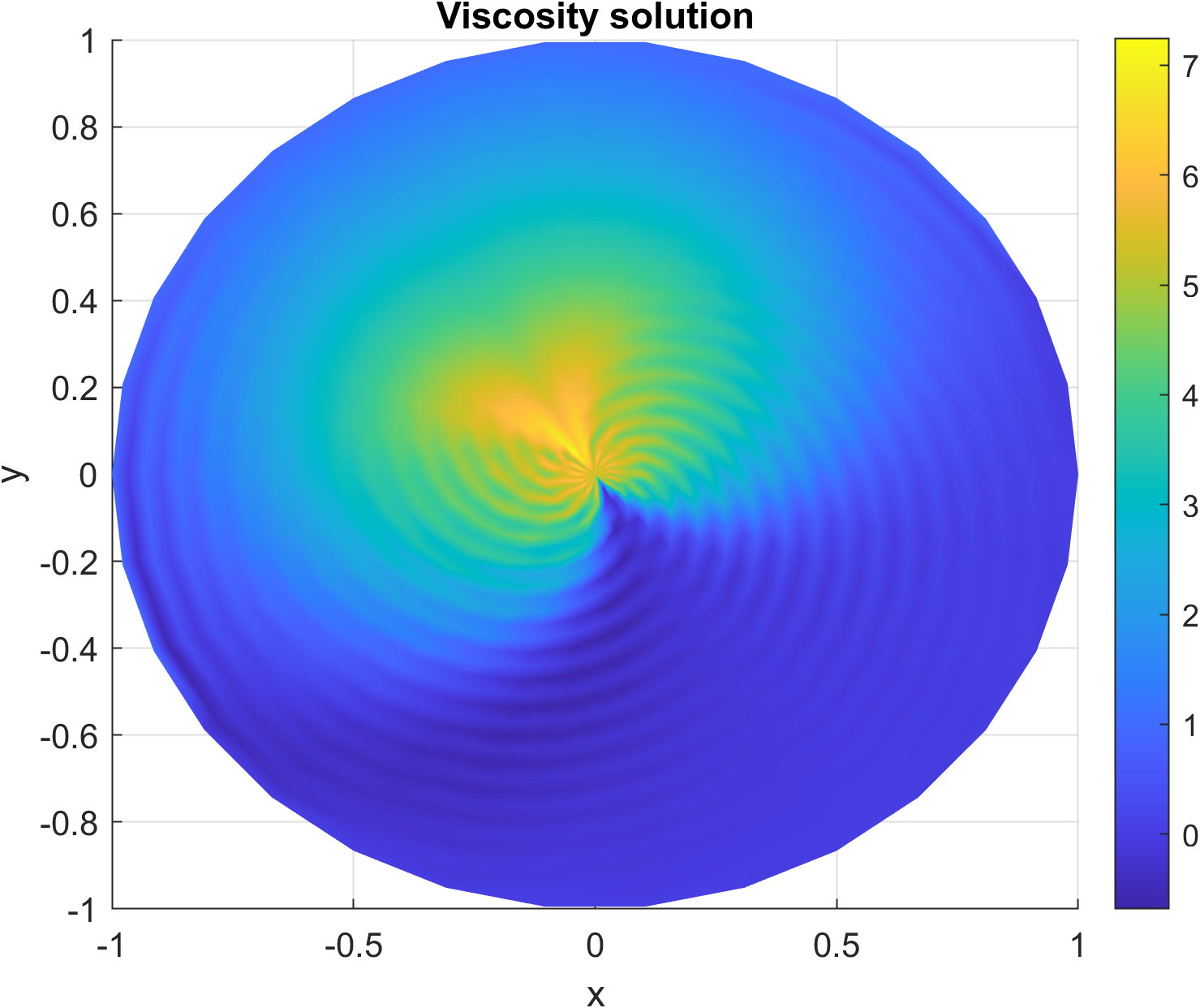}\qquad
   \includegraphics[scale=0.37]{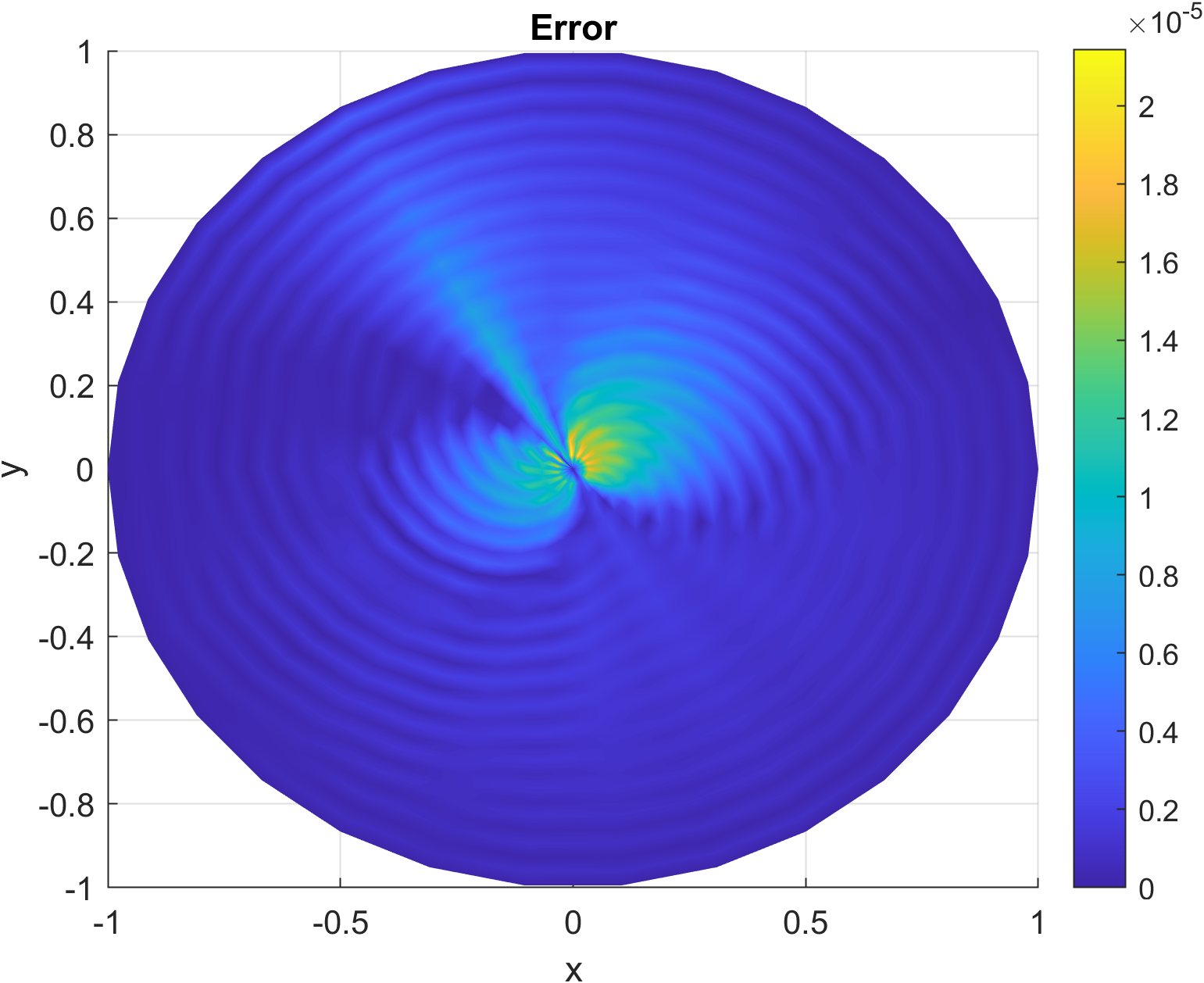}
  \caption{Solution of transport equation, viscosity equation and the relative error for $(I,J,K)=(30,30,10)$ and $\varepsilon=10^{-9}$}
   \end{center}
\end{figure} 
The figures (2)-(4) are computed by a finite difference method. We see that the smaller $\varepsilon$ gets the smaller the relative error becomes in each grid point. We might guess that the viscosity solution converges numerically to the transport solution as $\varepsilon\rightarrow 0$ for other choices of $f,n$ and $\alpha$. 

\section{Conclusions}


The characterization of tensor field tomography as inverse source problem for a transport equation is not new but offers an intriguing possibility to handle these problems for fairly general settings, i.e., for static as well as time-depending tensor fields of arbitrary rank $m$ in a medium with absorption and refraction, in a unified framework. This article builds the theoretical basis for solving the inverse problems by 
\begin{itemize}
\item defining the forward operators in mathematical settings that are relevant for applications
\item proving well-definedness of the operators by transferring to viscosity solutions
\end{itemize}
Any regularization method, be it variational or iterative, can now rely on these findings. Constructing and numerical implementation of such solvers as well as analytic investigations for $\varepsilon\to 0$ are subject of current research.




\begin{appendix}
\section{Proof of proposition \ref{prop}}\label{prop_proof}
\begin{proof}
Using \eqref{christ_refractive} and \eqref{surface_element} we write the left-hand side of \eqref{essential} as

\begin{align*}
-\int_{\Omega_x M}\Gamma_{ij}^{k}\xi_i \xi_j \frac{\partial u}{\partial \xi_k} u \mathrm{d}\omega_x(\xi) 
= \int_{0}^{\pi}\int_{0}^{2\pi}n^{-1}(x)\left(n^{-2}(x)\frac{\partial n}{\partial x_k}(x)-2\xi_k \langle \xi,\nabla n(x)\rangle \right)\frac{\partial u(x,\xi)}{\partial \xi_k}u(x,\xi)\sin\theta\mathrm{d}\varphi\mathrm{d}\theta.\\
\end{align*} 

Next we use \eqref{xi1}-\eqref{xi3} and obtain for $k=1,2,3$ separately

\begin{align*}
\left(n^{-2}(x)\frac{\partial n}{\partial x_1}(x)-2\xi_1 \langle \xi,\nabla n\rangle \right)\frac{\partial u}{\partial \xi_1}\sin\theta
&=\left(n^{-1}(x)\frac{\partial u}{\partial \theta}\right) \left(\partial_1 n\cos\varphi \cos\theta\sin\theta - 2\partial_1 n\cos^3 \varphi \cos\theta\sin^3\theta\right. \\&\left. - 2\partial_2 n \cos^2\varphi \sin\varphi \cos\theta\sin^3\theta - 2 \partial_3 n \cos^2\varphi \cos^2 \theta \sin^2\theta\right)\\
&+ \left(n^{-1}(x)\frac{\partial u}{\partial\varphi}\right)\left(-\partial_1 n \sin\varphi +2\partial_1 n \cos^2\varphi\sin\varphi \sin^2\theta \right. \\&\left. +2\partial_2 n \cos\varphi \sin^2\varphi \sin^2\theta + 2\partial_3 n\cos\varphi\sin\varphi\cos\theta\sin\theta \right)
\end{align*}

\begin{align*}
\left(n^{-2}(x)\frac{\partial n}{\partial x_2}(x)-2\xi_2 \langle \xi,\nabla n\rangle \right)\frac{\partial u}{\partial \xi_2}\sin\theta
&=\left(n^{-1}(x)\frac{\partial u}{\partial \theta}\right) \left(\partial_2 n\sin\varphi \cos\theta\sin\theta - 2\partial_1 n\cos \varphi\sin^2\varphi \cos\theta\sin^3\theta\right. \\&\left. - 2\partial_2 n \sin^3\varphi \cos\theta\sin^3\theta - 2 \partial_3 n \sin^2\varphi \cos^2 \theta \sin^2\theta\right)\\
&+ \left(n^{-1}(x)\frac{\partial u}{\partial\varphi}\right)\left(\partial_2 n \cos\varphi -2\partial_1 n \cos^2\varphi\sin\varphi \sin^2\theta \right. \\&\left. -2\partial_2 n \cos\varphi \sin^2\varphi \sin^2\theta - 2\partial_3 n\cos\varphi\sin\varphi\cos\theta\sin\theta \right)
\end{align*}

\begin{align*}
\left(n^{-2}(x)\frac{\partial n}{\partial x_3}(x)-2\xi_3 \langle \xi,\nabla n\rangle \right)\frac{\partial u}{\partial \xi_3}\sin\theta
&=\left(n^{-1}(x)\frac{\partial u}{\partial \theta}\right)\left( - \partial_3 n\sin^2\theta + 2\partial_1 n\cos\varphi\cos\theta\sin^3\theta \right. \\&\left.+ 2\partial_2 n \sin\varphi \cos\theta\sin^3\theta + 2\partial_3 n \cos^2\theta\sin^2\theta\right).
\end{align*}


After some simplifications we get

\begin{align*}
&\left(n^{-2}(x)\frac{\partial n}{\partial x_k}(x)-2\xi_k \langle \xi,\nabla n(x)\rangle \right)\frac{\partial u}{\partial \xi_k}\sin\theta\\
&=n^{-1}(x)(\partial_1 n \sin\theta\cos\theta\cos\varphi + \partial_2 n \sin\theta\cos\theta\sin\varphi - \partial_3 n\sin^2\theta)\frac{\partial u}{\partial\theta}\\&+ n^{-1}(x)(\partial_2 n\cos\varphi - \partial_1 n\sin\varphi)\frac{\partial u}{\partial\varphi}\\
\end{align*}

and thusly

\begin{align*}
&\quad\int_{0}^{\pi}\int_{0}^{2\pi} n^{-1}(x)\left(n^{-2}(x)\frac{\partial n}{\partial x^i}(x)-2\xi_k \langle \xi,\nabla n(x)\rangle \right)\frac{\partial u(x,\xi)}{\partial \xi_k}u(x,\xi)\sin\theta\mathrm{d}\varphi\mathrm{d}\theta\\
&=\int_{0}^{\pi}\int_{0}^{2\pi}n^{-2}(x) (\partial_1 n \sin\theta\cos\theta\cos\varphi + \partial_2 n \sin\theta\cos\theta\sin\varphi - \partial_3 n\sin^2\theta)\frac{\partial u}{\partial\theta}u\mathrm{d}\varphi\mathrm{d}\theta\\
&+\int_{0}^{\pi}\int_{0}^{2\pi}n^{-2}(x)(\partial_2 n\cos\varphi - \partial_1 n\sin\varphi)\frac{\partial u}{\partial\varphi}u\mathrm{d}\varphi\mathrm{d}\theta.\\
\end{align*}

An integration by parts in the first integral with respect to $\theta$ leads to

\begin{align*}
&\int_{0}^{2\pi}[n^{-2}(x) (\partial_1 n \sin\theta\cos\theta\cos\varphi + \partial_2 n \sin\theta\cos\theta\sin\varphi - \partial_3 n\sin^2\theta)u^2]_{0}^{\pi}\mathrm{d}\varphi\\
&- \int_{0}^{\pi}\int_{0}^{2\pi}n^{-2}(x) (\partial_1 n \sin\theta\cos\theta\cos\varphi + \partial_2 n \sin\theta\cos\theta\sin\varphi - \partial_3 n\sin^2\theta)\frac{\partial u}{\partial\theta}u\mathrm{d}\varphi\mathrm{d}\theta\\
&-\int_{0}^{\pi}\int_{0}^{2\pi}n^{-2}(x)  (\partial_1 n\cos\varphi(\cos^2\theta - \sin^2\theta) + \partial_2 n \sin\varphi(\cos^2\theta - \sin^2\theta)-2\partial_3 n \sin\theta\cos\theta)u^2     \mathrm{d}\varphi\mathrm{d}\theta\\
\end{align*}

leading to

\begin{align*}
& \int_{0}^{\pi}\int_{0}^{2\pi}n^{-2}(x) (\partial_1 n \sin\theta\cos\theta\cos\varphi + \partial_2 n \sin\theta\cos\theta\sin\varphi - \partial_3 n\sin^2\theta)\frac{\partial u}{\partial\theta}u\mathrm{d}\varphi\mathrm{d}\theta\\
&=- \frac{1}{2}\int_{0}^{\pi}\int_{0}^{2\pi}n^{-2}(x)  (\partial_1 n\cos\varphi(\cos^2\theta - \sin^2\theta) + \partial_2 n \sin\varphi(\cos^2\theta - \sin^2\theta)-2\partial_3 n \sin\theta\cos\theta)u^2     \mathrm{d}\varphi\mathrm{d}\theta.\\
\end{align*}

An according integration by parts with respect to $\varphi$ yields

\begin{align*}
&\quad \int_{0}^{\pi}[n^{-2}(x)\left(\partial_2 n(x) \cos\varphi -\partial_1 n(x) \sin\varphi \right)u^2]_{0}^{2\pi}\mathrm{d}\theta\\ 
&- \int_{0}^{\pi}\int_{0}^{2\pi}u \frac{\partial u}{\partial\varphi} n^{-2}(x)\left(\partial_2 n(x) \cos\varphi -\partial_1 n(x) \sin\varphi \right)\mathrm{d}\varphi\mathrm{d}\theta\\
&-\int_{0}^{\pi}\int_{0}^{2\pi} u^2 n^{-2}(x)\left(-\partial_2 n(x) \sin\varphi -\partial_1 n(x) \cos\varphi \right)\mathrm{d}\varphi\mathrm{d}\theta.\\
\end{align*}

The first summand vanishes and we get

\begin{align*}
\int_{0}^{\pi}\int_{0}^{2\pi}n^{-2}(x)(\partial_2 n\cos\varphi - \partial_1 n\sin\varphi)\frac{\partial u}{\partial\varphi}u\mathrm{d}\varphi\mathrm{d}\theta
&= - \frac{1}{2}\int_{0}^{\pi}\int_{0}^{2\pi} u^2 n^{-2}(x)\left(-\partial_2 n(x) \sin\varphi -\partial_1 n(x) \cos\varphi \right)\mathrm{d}\varphi\mathrm{d}\theta.\\
\end{align*}

Finally, we arrive at

\begin{align*}
&\quad\int_{\Omega_x M}n^{-1}(x)\left(n^{-2}(x)\frac{\partial n}{\partial x_k}(x)-2\xi_k \langle \xi,\nabla n(x)\rangle\right)\frac{\partial u}{\partial \xi_k}u\mathrm{d}\omega_x(\xi)\\
&= \int_{0}^{\pi}\int_{0}^{2\pi} \frac{1}{2}u^2\cdot n^{-2}(x)(\partial_2 n(x) \sin\varphi +\partial_1 n(x) \cos\varphi )\mathrm{d}\varphi\mathrm{d}\theta \\
&+  \int_{0}^{\pi}\int_{0}^{2\pi}\frac{1}{2}u^2\cdot n^{-2}(x) (\partial_1 n\cos\varphi(\sin^2\theta - \cos^2\theta) + \partial_2 n \sin\varphi(\sin^2\theta - \cos^2\theta)+2\partial_3 n \sin\theta\cos\theta)\mathrm{d}\varphi\mathrm{d}\theta\\
&= \int_{\Omega_x M}n^{-1}\langle \nabla n, \xi\rangle  u^2 \mathrm{d}\omega_x (\xi).\\
\end{align*}

\end{proof}

\end{appendix}

\bibliography{references} 

\end{document}